\documentclass[12pt,english,a4paper,oneside]{amsart} 

\pdfoutput=1

\usepackage[english]{babel} 
\usepackage{a4} 
\usepackage{graphicx} 
\usepackage{pictexwd} 
\usepackage{dcpic} 
\usepackage[colorlinks=true,breaklinks,bookmarksopen,bookmarksnumbered]{hyperref} 
\hypersetup{pdfauthor=Katharina Ludwig}
\hypersetup{pdftitle=Spin curves}

\theoremstyle{plain}
\begingroup

\newtheorem{thm}{Theorem}
\newtheorem{cor}[thm]{Corollary}

\newtheorem{prop}[thm]{Proposition}

\numberwithin{thm}{section}
\endgroup

\newtheorem*{thm*}{Theorem}
\newtheorem*{conj*}{Conjecture}
\newtheorem*{verm*}{Vermutung}

\theoremstyle{definition}
\newtheorem{defn}[thm]{Definition}
\newtheorem{rem}[thm]{Remark}

\newtheorem{notation}[thm]{Notation}

\numberwithin{equation}{section}

\newcommand{\base}{Z}
\newcommand{\basept}{z}

\newcommand{\calB}{\mathcal{B}}
\newcommand{\calC}{\mathcal{C}}

\newcommand{\calL}{\mathcal{L}}

\newcommand{\calO}{\mathcal{O}}

\newcommand{\calX}{\mathcal{X}}

\newcommand{\CC}[1][]{\mathbb{C}^{#1}}

\newcommand{\comp}{C_j}
\newcommand{\compi}{C_{j'}}
\newcommand{\compii}{C_{j''}}

\newcommand{\compnu}{C_j^\nu}
\newcommand{\compinu}{C_{j'}^\nu}

\newcommand{\cspin}{\left(\cspini,\cspinii,\cspiniii\right)}
\newcommand{\cspini}{f:\calX\rightarrow\base}
\newcommand{\cspinii}{\calL}
\newcommand{\cspiniii}{\calB}
\newcommand{\DC}{D}

\newcommand{\defspace}{B}

\newcommand{\ET}{T}
\newcommand{\EX}{\overline{N}}

\newcommand{\fibre}[1][\stab]{\overline{S}_{#1}}

\newcommand{\iso}{(\isoi,\isoii)}
\newcommand{\isoi}{\sigma}
\newcommand{\isoii}{\gamma}

\newcommand{\mgbar}[1][g]{\overline{M}_{#1}}

\newcommand{\nonex}{\widetilde{\spini}}
\newcommand{\nonexcomp}{\widetilde{\spini}_j}
\newcommand{\nonexcompi}{\widetilde{\spini}_{j'}}

\newcommand{\osquare}[1]{{#1}^{\otimes 2}}
\newcommand{\paut}{\underline{\aut}}

\newcommand{\PP}[1][]{\mathbb{P}^{#1}}

\newcommand{\rsbt}{Reid-Shepherd-Barron-Tai}
\newcommand{\RT}[1][3g-3]{\frac{1}{n}\sum_{j=1}^{#1}a_j}

\newcommand{\sgbar}[1][g]{\overline{S}_{#1}}
\newcommand{\spin}{(\spini,\spinii,\spiniii)}

\newcommand{\spinii}{L}
\newcommand{\spini}{X}
\newcommand{\spiniii}{b}
\newcommand{\stab}{C}

\newcommand{\ZZ}[1][]{\mathbb{Z}^{#1}}
\newcommand{\1}{\mathbb{I}}

\DeclareMathOperator{\aut}{Aut}
\DeclareMathOperator{\auto}{Aut_0}

\DeclareMathOperator{\eig}{Eig}
\DeclareMathOperator{\ext}{Ext}

\DeclareMathOperator{\fix}{Fix}
\DeclareMathOperator{\GL}{GL}

\DeclareMathOperator{\identity}{id}
\DeclareMathOperator{\im}{im}

\DeclareMathOperator{\lcm}{lcm}

\DeclareMathOperator{\ord}{ord}

\DeclareMathOperator{\sext}{\mathcal{E}\!\mathit{xt}}

\DeclareMathOperator{\sing}{sing}

\newcounter{mylist}
\newenvironment{mylist}%
{\begin{list}%
	{(\roman{mylist})}%
	{\usecounter{mylist}%
	 \setlength{\rightmargin}{0pt}%
	 \setlength{\leftmargin}{0pt}%
	 \setlength{\itemindent}{0.5em}%
	 \setlength{\itemsep}{0pt}%
	 \setlength{\parsep}{0ex plus0.1ex}
	 \setlength{\topsep}{0ex}}}%
{\end{list}}

\begin{document}
\setlength{\unitlength}{1cm}
\author{Katharina Ludwig}
\title{On the geometry of the moduli space of spin curves}
\begin{abstract}
We determine the smooth locus and the locus of canonical singularities in the Cornalba compactification $\sgbar$ of the moduli space $S_g$ of spin curves, i.e., smooth curves of genus $g$ with a theta characteristic. Moreover, the following lifting result for pluricanonical forms is proved: Every pluricanonical form on the smooth locus of $\sgbar$ extends holomorphically to a desingularisation of $\sgbar$.
\end{abstract}
\maketitle
\section{Introduction}
The moduli space $\sgbar$ constructed by M.~Cornalba in~\cite{co1989} compactifies the moduli space $S_g$ of smooth spin curves (over $\CC$). These are pairs $(\stab,\spinii)$ of a smooth curve of genus $g\geq 2$ and a theta characteristic $\spinii$ on $\stab$. Points in the boundary of this compactification correspond to certain line bundles on nodal curves and these lie naturally over points in the boundary of the Deligne-Mumford compactification $\mgbar$ of the moduli space $M_g$ of smooth curves of genus $g$ (see~\cite{demu1969}). In particular, there exists a natural morphism $\pi:\sgbar\rightarrow\mgbar$ which sends the moduli point of a spin curve to the moduli point of the underlying curve. The map $\pi$ is a ramified cover of degree $2^{2g}$.

Several authors constructed compactifications of $S_g$ and the moduli spaces of higher spin curves, where pointed smooth curves together with an $r$th root of a suitably twisted canonical bundle are considered. T.~J.~Jarvis in~\cite{ja1998,ja2000} included stable curves with certain rank $1$ torsion-free sheaves into the moduli problem. In~\cite{cacaco2004} L.~Caporaso, C.~Casagrande and Cornalba generalised the approach via line bundles on nodal curves to higher spin curves. Other compactifications were given by D.~Abramovich and Jarvis in~\cite{abja2003} and by A.~Chiodo in his thesis~\cite{ch2003} by means of line bundles on ``twisted curves'', these are, roughly spoken, stable curves with a stack structure at some nodes. In the case of ``ordinary'' spin curves all these different approaches give moduli spaces isomorphic to Cornalba's $\sgbar$. In 1991 Witten conjectured that certain intersection numbers on the moduli spaces of $r$-spin curves can be arranged into a power series that satisfies the $r$-KdV (or $r$th higher Gelfand Dikii) hierarchy of partial differential equations. This conjecture was proved by C.~Faber, S.~Shadrin and D.~Zvonkine in~\cite{fashzv2006}.

In a different direction one can ask for which values of $g$ the connected components $\sgbar^+$ and $\sgbar^-$ of $\sgbar$ are of general type, where $\sgbar^+$ resp. $\sgbar^-$ is the irreducible moduli space of even resp. odd spin curves. It is clear that for all $g$ such that $\mgbar$ is of general type $\sgbar^\pm$ are also because of the finite morphisms $\pi^\pm:\sgbar^\pm\rightarrow\mgbar$. Hence by the results of J.~Harris, D.~Mumford and D.~Eisenbud~\cite{eiha1987,hamu1982} and G.~Farkas~\cite{faM22} $\sgbar^\pm$ is of general type for $g\geq 24$ and $g=22$. One important ingredient for these results is the fact that for $g\geq 4$ every pluricanonical form on $\mgbar^0$, the locus of curves with trivial automorphism group, extends to a desingularisation $\widetilde{M}_g$ of $\mgbar$ (see~\cite{hamu1982}). This lifting result implies that in order to determine the Kodaira dimension of $\mgbar$ it is enough to understand the spaces of global pluricanonical forms on $\mgbar$ and one does not have to concern oneself with the desingularisation. We will prove an analogous result for pluricanonical forms on the smooth locus of $\sgbar$. About the question for which $g$ the moduli space $\sgbar^\pm$ is rational or unirational little seems to be known. G.~Bini and C.~Fontanari prove in their article~\cite{bifo2004} that the moduli space of even $n$-pointed spin curves of genus $1$ is unirational for $n\leq 10$ and has Kodaira dimension $1$ for $n\geq 12$.

Sections~2 and~3 of this article focus on the local (analytic) structure of the moduli space $\sgbar$. As in the case of $\mgbar$ an analytic neighbourhood of the moduli point of a spin curve $\spin$ in $\sgbar$ is isomorphic to the quotient $V/G$ of a $3g-3$-dimensional vector space $V$ with respect to a finite group $G$. This group is essentially the automorphism group of the spin curve under consideration. We give a criterion for the smoothness of the point $[\spin]\in\sgbar$ in terms of geometrical properties of the spin curve $\spin$. In Section~3 a detailed analysis of the occurring quotients gives a description of the locus of canonical singularities of $\sgbar$ with the help of the \rsbt{} criterion. The morphism $\pi$ plays an important role in these calculations, since it establishes a connection between the well understood singularities of $\mgbar$ (see~\cite{hamu1982}) and those of $\sgbar$. 

These local results are then used in Section~4 to prove that all pluricanonical forms on the smooth locus $\sgbar^{\text{reg}}$, i.e.,~sections in $\Gamma(\sgbar^{\text{reg}},\calO_{\sgbar}(kK_{\sgbar}))$, extend holomorphically to a desingularisation $\widetilde S_g$ of $\sgbar$. An important ingredient is the analogous result for $\mgbar$ by Harris and Mumford.

The results explained in this article are those of my PhD thesis. I am indebted to my advisor K.~Hulek for his guidance and encouragement. My thanks also go to L.~Caporaso and B.~Hassett for very helpful discussions.

\section{The non-singular locus of $\sgbar$}
In this section the coarse moduli space $\sgbar$ of spin curves constructed by Maurizio Cornalba in his article~\cite{co1989} and its non-singular locus $\sgbar^{\text{reg}}$ will be described. 
\begin{defn}
\begin{mylist}
\item A curve $\spini$ of (arithmetic) genus $g\geq 2$ is a \emph{quasistable curve} if it is the blow up $\beta:\spini\rightarrow\stab$ of a stable curve $\stab$ of genus $g$ at a set $N\subset\sing\stab$. A rational component $E$ in $\spini$ with $\beta(E)=P\in N$ is called an \emph{exceptional component}. A node in $N$ is an \emph{exceptional node}, while a node in $\Delta=\sing\stab\setminus N$ is a \emph{non-exceptional node}.
\item Let $\spini$ be a quasistable curve. The \emph{non-exceptional subcurve} $\nonex$ of $\spini$ is defined as
\[
\nonex=\overline{\spini\setminus\bigcup E},
\] 
where the union is taken over all exceptional components $E$ of $\spini$. 
\end{mylist}
\end{defn}
\begin{rem}
Let $\beta:\spini\rightarrow\stab$  be the blow up of the stable curve $\stab$ of genus $g\geq 2$ at $N\subset\sing\stab$. Note that the restriction $\widetilde\beta=\beta_{|\nonex}:\nonex\rightarrow\stab$ is the partial normalisation of $\stab$ at $N$. Moreover, the set of irreducible components of $\stab$ and the set of \emph{non-exceptional components} of $\spini$, i.e. irreducible components of $\nonex$, are in $1:1$-correspondence. For an irreducible component $\comp$ of $\stab$ the preimage $\widetilde\beta^{-1}(\comp)$ is a partial normalisation of $\comp$. Nevertheless, we will continue to denote $\widetilde\beta^{-1}(\comp)$ by $\comp$.
\end{rem}
\begin{defn}
\begin{mylist}
\item A \emph{spin curve} of genus $g\geq 2$ is a triple $\spin$, where $\spini$ is a quasistable curve of genus $g$ with stable model $\beta:\spini\rightarrow\stab$, $\spinii$ is a line bundle on $\spini$ and $\spiniii:\osquare{\spinii}\rightarrow\beta^*\omega_\stab$ is a homomorphism, such that the restriction $\spinii_{|E}$ to any exceptional component $E$ of $\spini$ is isomorphic to $\calO_E(1)$ and the restriction of $\spiniii$ to the non-exceptional subcurve $\nonex$ induces an isomorphism
\[
\widetilde\spiniii:\osquare{\spinii}_{|\nonex}\overset{\cong}{\longrightarrow}\omega_{\nonex}.
\]
\item A spin curve is \emph{even} resp. \emph{odd} if the dimension $h^0(\spini,\spinii)$ of the space of global sections of $\spinii$ is even resp. odd.
\item A \emph{family of spin curves} over a scheme $\base$ is a triple $\cspin$, where $\cspini$ is a flat family of quasistable curves, $\cspinii$ is a line bundle on $\calX$ and $\cspiniii:\osquare\cspinii\rightarrow\beta^*\omega_{\calC/\base}$ is a homomorphism, where $\beta:\calX\rightarrow\calC$ is the stable model and $\omega_{\calC/\base}$ is the relative dualizing sheaf of the family $f_\calC:\calC\rightarrow\base$ of stable curves, such that for every closed $\basept\in\base$ the restriction $\left(\calX_\basept,\calL_{|\calX_\basept},\calB_{|\calX_\basept}\right)$ to the fibre $\calX_\basept$ of $f$ over $\basept$ is a spin curve.
\item Let $\cspin$ and $\left(f':\calX'\rightarrow\base,\cspinii',\cspiniii'\right)$ be families of spin curves over $\base$.
As in~\cite{co1991} an \emph{isomorphism} between the two is a pair $\iso$ where $\isoi:\calX\rightarrow\calX'$ and $\isoii:\isoi^*\cspinii'\rightarrow\cspinii$ are isomorphisms over $\base$ such that the diagram
\[\begindc{0}[1]
\obj(0,0)[ul]{$\isoi^*{\beta'}^*\omega_{\calC'/\base}$} 
\obj(70,0)[ur]{$\beta^*\omega_{\calC/\base}$} 
\obj(0,40)[ol]{$\osquare{(\isoi^*\cspinii')} $} 
\obj(70,40)[or]{$\osquare{\cspinii}$} 

\mor{ul}{ur}{\footnotesize$\delta$} 
\mor{ol}{or}{\footnotesize$\osquare\isoii$} 
\mor{ol}{ul}{\footnotesize$\isoi^*\cspiniii'$} 
\mor(67,40)(67,0){\footnotesize$\cspiniii$} 
\enddc\]
commutes, where $\delta$ is the canonical isomorphism. In case $\calX$ and $\calX'$ have the same stable model $\calC$ and $\isoi:\calX\rightarrow\calX'$ is an isomorphism over $\calC$ the isomorphism $\iso$ is \emph{inessential}. 
The group of automorphisms resp. inessential automorphisms of $\cspin$ is denoted by $\aut\cspin$ resp.~$\auto\cspin$.
\end{mylist}
\end{defn}
\begin{rem}\label{rem:even}
Let $\stab$ be a stable curve of genus $g\geq 2$, $N\subset\sing\stab$ a subset and $\beta:\spini\rightarrow\stab$ the blow up at $N$. Out of degree reasons $\spini$ is the \emph{support} of a spin curve, i.e. there exists a spin curve $\spin$, if and only if the set of non-exceptional nodes $\Delta=\sing\stab\setminus N$ is \emph{even}, i.e. for every irreducible component $\comp$ of $\stab$ the degree of $\nu_\stab^{-1}(\Delta)\cap\compnu$ considered as a divisor $D_j$ on the normalisation $\compnu$ is even, where $\nu_\stab:\stab^\nu\rightarrow\stab$ is the normalisation of $\stab$ (see~\cite[p. 566]{co1989}). In order to construct a spin structure $(\spinii,\spiniii)$ on $\spini$ we have to choose a line bundle $\spinii^\nu_j$ on the normalisation $\compnu$ of every non-exceptional component $\comp$ of $\spini$ such that there exists an isomorphism
\[
\spiniii_j:\osquare{\spinii^\nu_j}\rightarrow\omega_{\compnu}(D_j).
\]
Caporaso and Casagrande prove in~\cite{caca2003} that there are $2^{b_1(\Gamma(\nonex))}$ gluings of these line bundles which give non-isomorphic line bundles $\widetilde\spinii$ on the non-exceptional subcurve $\nonex$ such that the $\spiniii_j$ glue to an isomorphism $\widetilde\spiniii:\osquare{\widetilde\spinii}\rightarrow\omega_{\nonex}$. Here $b_1(\Gamma(\nonex))$ is the first Betti number of the dual graph $\Gamma(\nonex)$ of the nodal curve $\nonex$ which has a vertex for every irreducible component and an edge for every node which is incident to the vertices corresponding to the components in which the branches at the node lie. Any gluing of $\widetilde\spinii$ to an $\calO_E(1)$ for every exceptional component $E$ gives a line bundle $\spinii$ on $\spini$. Extending $\widetilde\spiniii$ to a homomorphism $\spiniii:\spinii\rightarrow\beta^*\omega_\stab$ by $0$ on the exceptional components gives a spin structure $(\spinii,\spiniii)$. Moreover, Caporaso and Casagrande show that all possible gluings in this last step give raise to isomorphic spin curves.
\end{rem}
There exists a coarse moduli space $\sgbar$ for spin curves of genus $g\geq 2$. Sending the moduli point $[\spin]\in\sgbar$ of a spin curve $\spin$ to the moduli point $[\stab]\in\mgbar$ of the stable model $\stab$ of its support $\spini$ gives a map $\pi:\sgbar\rightarrow\mgbar$.
\begin{thm} \emph{\cite[Proposition 5.2. and Lemma 6.3.]{co1989}}
The variety $\sgbar$ is normal and projective and contains the moduli space $S_g$ of smooth spin curves as a dense subvariety. Moreover, $\sgbar$ consists of two connected components $\sgbar^+$ and $\sgbar^-$, the moduli spaces of even resp. odd spin curves. The components $\sgbar^+$ and $\sgbar^-$ are irreducible. The map $\pi$ is a finite morphism of degree $2^{2g}$.
\end{thm}

As in the case of $\mgbar$ the local analytic structure of $\sgbar$ at a point $[\spin]$ can be described as a quotient of the base of the local universal deformation of $\spin$ modulo the action of the automorphism group $\aut\spin$. Recall (see e.g.~\cite{hamo1998}) that the local universal deformation of a stable curve $\stab$ of genus $g\geq 2$ is (the germ at $0\in\defspace$ of) a family $\calC\rightarrow\defspace$, where $\defspace$ is smooth and $3g-3$-dimensional. The action of the automorphism group $\aut\stab$ on the special fibre $\stab$ of $\calC\rightarrow\defspace$ over $0$ extends to a fibre-preserving action on $\calC$, hence $\aut\stab$ acts on $\defspace$. Since the automorphism group is finite, its action can be linearized in suitable coordinates giving a representation $\aut\stab\rightarrow\GL(T_{0,\defspace})$ of $\aut\stab$ on the tangent space $T_{0,\defspace}$ and $\defspace/\aut\stab\cong T_{0,\defspace}/\aut\stab$ (in suitable neighbourhoods of $0\in\defspace$ and $0\in T_{0,\defspace}$). 

The tangent space $T_{0,\defspace}$ is the space of infinitesimal deformations of $\stab$ which is $\ext^1\left(\Omega_\stab^1,\calO_\stab\right)$, where $\Omega^1_\stab$ is the sheaf of K\"ahler differentials. Consider the normalisation $\nu_\stab:\stab^\nu\rightarrow\stab$ and denote by $P_i^\pm$ the two preimages of the node $P_i$ under $\nu_\stab$. For every irreducible component $\comp$ define $D_j$ to be $\{P_i^\pm|P_i\in\sing\stab\}\cap\compnu$ considered as a divisor on $\compnu$. Then there is a short exact sequence
\[
0\rightarrow \bigoplus_{\comp}H^1\left(\compnu,T_{\compnu}\left(D_j\right)\right) 
\rightarrow \ext^1\left(\Omega_\stab^1,\calO_\stab\right)
\rightarrow \bigoplus_{P_i}\sext^1\left(\Omega_\stab^1,\calO_\stab\right)_{P_i}
\rightarrow 0,
\]
where $H^1\left(\compnu,T_{\compnu}\left(D_j\right)\right)$ is the $3g(\compnu)-3+\deg D_j$-dimensional space of infinitesimal deformations of the pointed curve $(\compnu,\{P_i^\pm\}\cap\compnu)$ (with an arbitrary but fixed ordering of the points). On the right hand side $\sext^1\left(\Omega_\stab^1,\calO_\stab\right)_{P_i}$ is the space of infinitesimal deformations of the node $P_i$, i.e. if $xy=0$ is a local equation for $\stab$ at $P_i$ then $xy=t_i$ is the local universal deformation of the node and $t_i$ is a coordinate for $\sext^1\left(\Omega_\stab^1,\calO_\stab\right)_{P_i}\cong\CC$. Now choose coordinates $t_1,\dotsc,t_{3g-3}$ of $\ext^1\left(\Omega_\stab^1,\calO_\stab\right)=\CC[3g-3]_t$ compatible with the above sequence, i.e. for $i=1,\dotsc,\#\sing\stab$ the coordinate $t_i$ corresponds to the node $P_i$ and for every component $\comp$ there is a subset of the coordinates which is a coordinate system of $H^1\left(\compnu,T_{\compnu}\left(D_j\right)\right)$. The action of an automorphism $\isoi_\stab\in\aut\stab$ then induces an isomorphism $H^1\left(\compnu,T_{\compnu}\left(D_j\right)\right)\overset{\cong}{\rightarrow}H^1\left(\compinu,T_{\compinu}\left(D_{j'}\right)\right)$ where $\compi=\isoi_\stab(\comp)$. Moreover, if $P_i$ is a node and $P_{i'}=\isoi_\stab(P_i)$ its image, then the action of $\isoi_\stab$ on $\CC[3g-3]_t$ maps the coordinate $t_i$ corresponding to $P_i$ to a nonzero scalar multiple $\bar{c}_it_{i'}$ of the coordinate $t_{i'}$ corresponding to $P_{i'}$.

The local universal deformation of a spin curve has been constructed by Cornalba~\cite[p.~569ff.]{co1989}, see also~\cite[Section 3.3.]{cacaco2004} and~\cite[Theorem 2.2.]{ja2000}, in the following way. Given a spin curve $\spin$ of genus $g\geq 2$ consider the stable model $\stab$ of its support $\spini$. Let $N\subset\sing\stab$ be the set of exceptional nodes and $\Delta=\sing\stab\setminus N$ that of non-exceptional ones. Consider the morphism $\CC[3g-3]_\tau\rightarrow\CC[3g-3]_t$ defined by $t_i=\tau_i^2$ if $P_i\in N$ and $t_i=\tau_i$ else and the pull back family $\calC'=\calC\times_{\CC[3g-3]_t}\CC[3g-3]_\tau\rightarrow\CC[3g-3]_\tau$. For each $P_i\in N$ there exists a section $\{\tau_i=0\}\rightarrow\calC'$ whose image consists entirely of nodes and passes through the node $P_i$ in the central fibre. If $\beta:\calX\rightarrow\calC'$ is the blow up of the images of all these sections, then $\calX\rightarrow\CC[3g-3]_\tau$ is a family of quasistable curves with central fibre isomorphic to $\spini$. Cornalba shows that (up to changing the identification of $\spini$ with the central fibre and shrinking the neighbourhood of $0$ under consideration) the line bundle $\spinii$ on the central fibre and the homomorphism $\spiniii:\osquare\spinii\rightarrow\beta^*\omega_\stab$ can be extended to a line bundle $\cspinii$ on $\calX$ and a homomorphism $\cspiniii:\osquare\cspinii\rightarrow\beta^*\omega_{\calC'/\CC[3g-3]_\tau}$ in a unique way. Then (the germ at $0$ of) the triple $(\calX\rightarrow\CC[3g-3]_\tau,\cspinii,\cspiniii)$ is a family of spin curves and it is the local universal deformation of $\spin$. The automorphism group of $\spin$ is finite and acts linearly on $\CC[3g-3]_\tau$ and $\sgbar$ locally at $[\spin]$ is isomorphic to $\CC[3g-3]_\tau/\aut\spin$ at $0$.

Moreover, if $\iso$ is an automorphism of $\spin$ and $\isoi_\stab$ the induced automorphism of the stable model $\stab$, the action of $\iso$ on $\CC[3g-3]_\tau$ is a lift of the action of $\isoi_\stab$ on $\CC[3g-3]_t$ via $\CC[3g-3]_\tau\rightarrow\CC[3g-3]_t$. Therefore, if $P_i$ is an exceptional node of $\stab$ and $\isoi_\stab$ acts as $t_i\mapsto \bar{c}_it_{i'}$ then $P_{i'}=\isoi_\stab(P_i)$ is also exceptional and $\iso$ acts as $\tau_i\mapsto c_i\tau_{i'}$ where $c_i^2=\bar{c}_i$. The action of $\iso$ on coordinates corresponding to non-exceptional nodes or  components is the same as the action of $\isoi_\stab$ on these coordinates.

\begin{rem}\label{rem:triv}
Note that, if $\spin$ is a general spin curve then its automorphism group is $\{(\identity_\spini,\pm\identity_\spinii)\}$. Therefore, if $\spin$ is any spin curve $(\identity_\spini,\pm\identity_\spinii)\in\aut\spin$, these two automorphisms act trivially on $\CC[3g-3]_\tau$ and 
\[
\CC[3g-3]_\tau/\aut\spin=\CC[3g-3]_\tau/\paut\spin,
\]
where we consider $\paut\spin=\aut\spin/\{(\identity_\spini,\pm\identity_\spinii)\}$ as a subgroup of $\GL(\CC[3g-3]_\tau)$.
\end{rem}

As a first step we want to study the automorphism group of the spin curve $\spin$. By the definition of inessential automorphisms the following sequence is exact
\begin{align*}
0\longrightarrow\auto\spin\longrightarrow\aut\spin&\longrightarrow\aut\stab\\
\iso&\longmapsto\isoi_\stab
\end{align*}
\begin{prop}\label{prop:aut} 
Let $\spin$ be a spin curve of genus $g\geq 2$ and $\isoi_\stab$ an automorphism of the stable model $\stab$ of $\spini$. The automorphism $\isoi_\stab$ is in the image of the homomorphism $\aut\spin\rightarrow\aut\stab$, if and only if $\isoi_\stab(N)=N$ and $\widetilde\isoi^*\widetilde\spinii\cong\widetilde\spinii$, where $\widetilde\isoi:\nonex\rightarrow\nonex$ is the unique induced automorphism of the non-exceptional subcurve and $\widetilde\spinii=\spinii_{|\nonex}$. Moreover, every $\isoi_\stab$ fulfilling these conditions has exactly $2^{\#CC(\nonex)}$ preimages, where $CC(\nonex)$ is the set of connected components of $\nonex$. 
\end{prop}
\begin{defn}
Let $\spin$ be a spin curve of genus $g\geq 2$ with stable model $\stab$. An automorphism $\isoi_\stab\in\aut\stab$ \emph{lifts to the spin curve $\spin$} if $\isoi_\stab$ lies in the image of the homomorphism $\aut\spin\rightarrow\aut\stab$.
\end{defn}
\begin{rem}
Note that in general the resulting short exact sequence
\begin{multline*}
0\rightarrow\auto\spin\rightarrow\aut\spin
\rightarrow \{\isoi_\stab\in\aut\stab|\isoi_\stab \text{ lifts to }\spin\} \rightarrow 0
\end{multline*}
does not split.
\end{rem}
\begin{proof}
It is obvious that the image $\isoi_\stab$ of an automorphism $\iso\in\aut\spin$ permutes the set of exceptional nodes and the restriction of $\isoii:\isoi^*\spinii\rightarrow\spinii$ to $\nonex$ gives an isomorphism between $\widetilde\isoi^*\widetilde\spinii$ and $\widetilde\spinii$.

Let $\isoi_\stab$ be an automorphism of $\stab$ such that $\isoi_\stab(N)=N$ and $\varphi:\widetilde\isoi^*\widetilde\spinii\overset{\cong}{\longrightarrow}\widetilde\spinii$ an isomorphism. For every connected component $\nonexcomp$ of $\nonex$ there exists a unique scalar $\eta_j$ such that the following diagram of isomorphisms commutes
\[\begindc{0}[1]
\obj(0,0)[ul]{$(\widetilde\isoi^*\omega_{\nonex})_{|\nonexcomp}$} 
\obj(70,0)[ur]{$\omega_{\nonexcomp}$} 
\obj(0,40)[ol]{$\osquare{(\widetilde\isoi^*\widetilde\spinii)}_{|\nonexcomp} $} 
\obj(70,40)[or]{$\osquare{\widetilde\spinii}_{|\nonexcomp}$} 

\mor{ul}{ur}{\footnotesize} 
\mor{ol}{or}{\footnotesize$\eta_j\osquare\varphi_{|\nonexcomp}$} 
\mor{ol}{ul}{\footnotesize$(\widetilde\isoi^*\widetilde\spiniii)_{|\nonexcomp}$} 
\mor{or}{ur}{\footnotesize$\widetilde\spiniii_{|\nonexcomp}$} 
\enddc\]
Therefore, there are exactly two isomorphisms $\widetilde\isoii_j:(\widetilde\isoi^*\widetilde\spinii)_{|\nonexcomp}\rightarrow \widetilde\spinii_{|\nonexcomp}$ such that $\osquare{\widetilde\isoii_j}$ makes the above diagram commutative, namely $\lambda_j\varphi_{|\nonexcomp}$ where $\lambda_j$ is one of the two roots of $\eta_j$. Hence there are exactly $2^{\#CC(\nonex)}$ isomorphisms $\widetilde\isoii:\widetilde\isoi^*\widetilde\spinii\rightarrow\widetilde\spinii$ compatible with the isomorphisms to the canoncial bundle. The proof of Lemma~2.3.2. in~\cite{cacaco2004} shows that for every such $\widetilde\isoii$ there exists a unique extension $\iso\in\aut\spin$ of $(\widetilde\isoi,\widetilde\isoii)$.
\end{proof}

In the description of the non-singular locus $\sgbar^{\text{reg}}$ of $\sgbar$ the following graph plays an important role.
\begin{defn}
Let $\spini$ be a quasistable curve. The graph $\Sigma(\spini)$ consists of one vertex $v(\nonexcomp)$ for every connected component $\nonexcomp$ of the non-exceptional subcurve $\nonex$ of $\spini$ and one edge $e(E_i)$ for every exceptional component $E_i$ of $\spini$. If $E_i$ meets the non-exceptional subcurve in the connected components $\nonexcomp$ and $\nonexcompi$, the edge $e(E_i)$ is incident to the vertices $v(\nonexcomp)$ and $v(\nonexcompi)$. In case $\nonexcomp=\nonexcompi$ the edge $e(E_i)$ is a loop. 
\end{defn}
\begin{rem}\label{rem:iness}
Let $\spin$ be a spin curve, $\Sigma(\spini)$ the above defined graph and denote by $V(\Sigma(\spini))$ its set of vertices. Then it follows from Proposition~\ref{prop:aut} that $\auto\spin$ is isomorphic to $\ZZ[{V(\Sigma(\spini))}]_2$. If $(\isoii_j)_j\in\ZZ[{V(\Sigma(\spini))}]_2$ is given, the corresponding inessential automorphism $\iso$ is determined by requiring that $\isoi_{|\nonex}=\identity_{\nonex}$ and $\isoii_{|\nonexcomp}:\spinii_{|\nonexcomp}\rightarrow\spinii_{|\nonexcomp}$ is $(-1)^{\isoii_j}\identity_{\spinii_{|\nonexcomp}}$, i.e. $\isoii$ is multiplication with $(-1)^{\isoii_j}$ in every fibre of $\spinii$ over $\nonexcomp$ (see~\cite{caca2003}).
\end{rem}
\begin{defn}
A graph $\Gamma$ is a \emph{tree} if $\Gamma$ is connected and the first Betti number $b_1(\Gamma)$ is zero. A graph $\Gamma$ is \emph{tree-like} if the graph obtained from $\Gamma$ by removing all loops is a tree.
\end{defn}
Moreover, elliptic tails of the stable model $\stab$ of the support $\spini$ give raise to special cases. 
\begin{defn}
Let $\stab$ be a stable curve of genus $g\geq 2$. An irreducible component $\comp$ of $\stab$ is an \emph{elliptic tail} if its arithmetic genus is $1$ and $\comp$ meets the rest of the curve in exactly one node $P$, which is then called an \emph{elliptic tail node}. The elliptic tail $\comp$ is \emph{smooth} if $\comp$ is a smooth elliptic curve. Otherwise $\comp$ is a rational curve with one node and the elliptic tail is called \emph{singular}. We choose the node $P$ as the origin of the elliptic curve. A non-trivial automorphism $\isoi_\stab\in\aut\stab$ is called an \emph{elliptic tail automorphism of order $n$ with respect to the elliptic tail $\comp$} if $\isoi_\stab$ is the identity on $\overline{\stab\setminus\comp}$ and ${\isoi_\stab}_{|\comp}$ has order $n$. 
\end{defn}

\begin{thm}\label{thm:smooth}
Let $\spin$ be a spin curve of genus $g\geq 4$ with stable model $\stab$. The moduli space $\sgbar$ is smooth at the point $[\spin]$ if and only if the following two conditions are fulfilled.
\begin{mylist}
\item The graph $\Sigma(\spini)$ is tree-like.
\item The subgroup $\{\isoi_\stab\in\aut\stab|\isoi_\stab\text{ lifts to }\spin\}$ is generated by elliptic tail automorphisms of order $2$.
\end{mylist}
\end{thm}
Since $\sgbar$ at the point $[\spin]$ is locally isomorphic to $\CC[3g-3]_\tau/\paut\spin$ at $0$, the point $[\spin]\in\sgbar$ is smooth if and only if $\paut\spin\subset\GL(\CC[3g-3]_\tau)$ is generated by \emph{quasireflections}, i.e. by elements in $\GL(\CC[3g-3]_\tau)$ having $1$ as an eigenvalue of multiplicity exactly $3g-4$ (see e.g.~\cite{pr1967}). 
\begin{prop}\label{prop:qrsg}
Let $\iso$ be an automorphism of a spin curve $\spin$ of genus $g\geq 4$ and $\isoi_\stab$ the induced automorphism on the stable model $\stab$. Then $\iso$ acts on $\CC[3g-3]_\tau$ as a quasireflection exactly in the following cases:
\begin{mylist}
\item $\iso$ is inessential and there exists an exceptional component $E$ of $\spini$ such that $\overline{\spini\setminus E}$ consists of two connected components $\spini_1$ and $\spini_2$ such that $\isoii$ is multiplication with $1$ resp. $-1$ in every fibre of $\spinii$ over $\spini_1$ resp. $\spini_2$.
\item There exists an elliptic tail $\comp$ of $\stab$ such that $\isoi_\stab$ is the elliptic tail automorphism of order $2$ with respect to $\comp$ and $\isoi$ restricted to $\overline{\spini\setminus(\comp\cup E_j)}$ is the identity, where $E_j$ is the exceptional component meeting $\comp$. 
\end{mylist}
\end{prop}
\begin{rem}\label{rem:qrmg}
Any automorphism $\isoi_\stab$ of a stable curve $\stab$ of genus at least $4$ acts on $\CC[3g-3]_t$ as a quasireflection if and only if $\isoi_\stab$ is an elliptic tail automorphism of order $2$. If $\comp$ is the elliptic tail such that ${\isoi_\stab}_{|\comp}$ is the elliptic involution, denote by $P_j$ the node on $\comp$ and by $t_j$ the corresponding coordinate. Then $\isoi_\stab$ acts as $t_j\mapsto -t_j$ and $t_i\mapsto t_i$ else. This follows easily from Theorem~2 in~\cite{hamu1982} and implies that the smooth locus $\mgbar^{\text{reg}}$ of $\mgbar$ is
\[
\mgbar^{\text{reg}}=\left\{[\stab]\in\mgbar\left|\parbox{33ex}{$\aut\stab$ is generated by elliptic tail automorphisms of order $2$}\right.\right\}
\]
\end{rem}
\begin{proof}[Proof of Proposition~\ref{prop:qrsg}]
Let $\iso$ be an automorphism acting as a quasireflection on $\CC[3g-3]_\tau$ and $\isoi_\stab$ the induced automorphism of the stable model $\stab$. The action of $\isoi_\stab$ on $\CC[3g-3]_t$ decomposes into an action on $\bigoplus_{\comp}H^1\left(\compnu,T_{\compnu}(D_j)\right)$ and an action on $\bigoplus_{P_i\in\sing\stab}\CC_{t_i}$. Since $\isoi_\stab$ fixes the set $N$ of exceptional nodes the second action decomposes into an action on $\bigoplus_{P_i\in N}\CC_{t_i}$ and one on $\bigoplus_{P_i\in\Delta}\CC_{t_i}$, where $\Delta=\sing\stab\setminus N$. The action of $\iso$ on $\CC[3g-3]_\tau$ decomposes accordingly. $\iso$ acts as a quasireflection, hence there is exactly one eigenvalue $\xi\neq1$. 

\emph{Case 1: $\xi$ is an eigenvalue of the action on $\bigoplus_{\comp}H^1\left(\compnu,T_{\compnu}(D_j)\right)$.} In particular the action on the coordinates corresponding to nodes is trivial. Therefore, $\isoi_\stab$ acts on the $t$-coordinates as $\iso$ does on the $\tau$-coordinates. This means that $\isoi_\stab$ acts trivially on $\bigoplus_{P_i\in\sing\stab}\CC_{t_i}$ and as a quasireflection on $\bigoplus_{\comp}H^1\left(\compnu,T_{\compnu}(D_j)\right)$, which contradicts Remark~\ref{rem:qrmg}.

\emph{Case 2: $\xi$ is an eigenvalue of the action on $\bigoplus_{P_i\in\Delta}\CC_{\tau_i}$.} The action on the coordinates corresponding to exceptional nodes is trivial as well as the action on the coordinates corresponding to components. For all nodes $P_i\in\Delta$ we have $\tau_i=t_i$. As above this implies that the action of $\isoi_\stab$ is the same as that of $\iso$. By Remark~\ref{rem:qrmg} there is a node $P_j$ on an elliptic tail $\comp$ and $t_j\mapsto -t_j$ while $t_i\mapsto t_i$ else. Since $\xi=-1$ is an eigenvalue of the action of $\isoi_\stab$ on $\bigoplus_{P_i\in\Delta}\CC_{t_i}$, the node $P_j$ is non-exceptional, i.e. $P_j\in\Delta$. But a node connecting an elliptic tail to the rest of the curve is exceptional by Remark~\ref{rem:even}, hence this case is impossible.

\emph{Case 3: $\xi$ is an eigenvalue of the action on $\bigoplus_{P_i\in N}\CC_{\tau_i}$.} Consider the action of $\isoi_\stab$ on the set $N$ of exceptional nodes and let $P_{i_0},P_{i_1}=\isoi_\stab(P_{i_0}),\dotsc,P_{i_{m-1}}=\isoi_\stab^{m-1}(P_{i_0})$ be pairwise distinct, while $\isoi_\stab^m(P_{i_0})=P_{i_0}$. The action of $\isoi_\stab$ on $\bigoplus_{P_i\in N}\CC_{t_i}$ restricts to an action on $\bigoplus_{k=0}^{m-1}\CC_{t_{i_k}}$, hence the action of $\iso$ restricts to an action on $\bigoplus_{k=0}^{m-1}\CC_{\tau_{i_k}}$. We may assume that the cycle of nodes under consideration is the unique one, such that this restricted action is non-trivial. It is then a quasireflection and $\xi$ is the only eigenvalue different from~$1$. 

The action of $\iso$ on $\bigoplus_{k=0}^{m-1}\CC_{\tau_{i_k}}$ is given by $\tau_{i_k}\mapsto c_{i_k}\tau_{i_{k+1}}$ for appropriate non-zero scalars $c_{i_k}$, where the index $k$ is considered modulo $m$. A straightforward calculation yields that the eigenvalues of this action are the $m$th roots of $\prod_{k}c_{i_k}$. But the eigenvalues are $\xi,1,\dotsc,1$, therefore, either $m=2$, $c_{i_0}c_{i_1}=1$ and the eigenvalues are $1$ and $-1$ or $m=1$ and the eigenvalue is $\xi=c_{i_0}$. In case $m=2$ the action of $\isoi_\stab$ on the coordinates $t_{i_0}$ and $t_{i_1}$ is $t_{i_0}\mapsto c_{i_0}^2t_{i_1}$ and $t_{i_1}\mapsto c_{i_1}^2t_{i_0}$, which has also eigenvalues $1$ and $-1$. Hence $\isoi_\stab$ acts as a quasireflection and interchanges the two nodes $P_{i_0}$ and $P_{i_1}$. This contradicts Remark~\ref{rem:qrmg}. Therefore, $m=1$ and the node $P_{i_0}$ is a fixed point of $\isoi_\stab$. The action of $\isoi_\stab$ is $t_{i_0}\mapsto c_{i_0}^2t_{i_0}$ and $t_i\mapsto t_i$ else.

By Remark~\ref{rem:qrmg} there are two possibilities, either $c_{i_0}^2=-1$ and $\isoi_\stab$ is an elliptic tail automorphism of order $2$ or $c_{i_0}^2=1$ and $\isoi_\stab$ is the identity. In the first case denote by $\comp$ the elliptic tail of $\stab$ meeting the rest of the curve in the node $P_{i_0}$. In order to show that this is case (ii) of the statement, we have to prove, that $\isoi$ restricted to $\overline{\spini \setminus (\comp \cup E_{i_0})}$ is the identity, where $E_{i_0}$ is the exceptional component of $\spini$ over $P_{i_0}$. Since $\isoi_\stab$ is the identity on $\overline{\stab\setminus\comp}$, $\isoi$ is the identity on every non-exceptional component $\compi\neq\comp$. Let $E_i\neq E_{i_0}$ be an exceptional component meeting $\spini$ in the non-exceptional components $\compi$ and $\compii$. Note that $E_i$ does not meet the elliptic tail $\comp$ and $\isoi$ is the identity on $\compi$ and $\compii$. Hence $\isoii$ is multiplication with $(-1)^{\isoii_{j'}}$ resp. $(-1)^{\isoii_{j''}}$ in all fibres of $\spinii$ over $\compi$ resp. $\compii$ for appropriate $\isoii_{j'},\isoii_{j''}\in\ZZ_2$. In this situation $\isoi$ is the identity on $E_i$ if and only if $\isoii_{j'}=\isoii_{j''}$. 

Since $\iso$ acts as a quasireflection and $E_i\neq E_{i_0}$ we have $\tau_i\mapsto\tau_i$, where $\tau_i$ is the coordinate corresponding to $E_i$. Consider the restriction of the universal deformation $(\calX\rightarrow\CC[3g-3]_\tau,\cspinii,\cspiniii)$ to the one-dimensional locus given by $\tau_k=0$, $k\neq i$. In the underlying one-parameter family of quasistable curves the subcurve $\compi\cup E_i \cup\compii$ of the central fibre $\spini$ is smoothed. $\tau_i\mapsto\tau_i$ means that the automorphism $\iso$ of the central fibre $\spin$ deforms to the nearby curves. This is the case if and only if $\isoii_{j'}=\isoii_{j''}$. Therefore, $\isoi$ is the identity on $\overline{\spini\setminus(\comp\cup E_{i_0})}$ and $\isoii$ is $\pm\identity$ over this subcurve.

{\scriptsize
\[\begindc{0}[1]
\obj(0,30)[]{central fibre}[1] 
\obj(120,30)[]{nearby fibre}[1] 
\obj(0,0)[central]{{\includegraphics{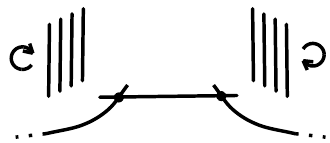}}}[1] 
\obj(120,0)[near]{{\includegraphics{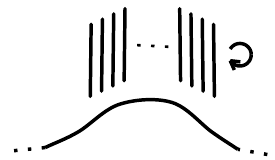}}}[1] 
\obj(-2,2){$E_i$}
\obj(-17,-20){$\compi$}
\obj(20,-20){$\compii$}
\obj(-60,3){$(-1)^{\isoii_{j'}}$}
\obj(65,4){$(-1)^{\isoii_{j''}}$}
\obj(193,7){$(-1)^{\isoii_{j'}}=(-1)^{\isoii_{j''}}$}
\enddc\]}

Consider now the case that $\isoi_\stab$ is the identity, i.e. $\iso$ is an inessential automorphism, which by Remark~\ref{rem:iness} corresponds to an element $(\isoii_j)_j\in\ZZ[V(\Sigma(\spini))]_2$. Let $E_i$ be an exceptional component, $\tau_i$ the corresponding coordinate and denote by $\nonexcomp$ and $\nonexcompi$ the connected components of the non-exceptional subcurve $\nonex$ meeting $E_i$. As above $\tau_i\mapsto\tau_i$ if and only if $\isoii_j=\isoii_{j'}$. Since every non-trivial inessential automorphism has order $2$ we must have $\tau_i\mapsto -\tau_i$ in case $\isoii_j\neq\isoii_{j'}$. Consider the subcurve $\overline{\spini\setminus E_{i_0}}$, where $E_{i_0}$ is the unique exceptional component such that $\tau_{i_0}\mapsto-\tau_{i_0}$, i.e. here $\isoii_j\neq\isoii_{j'}$. Assume that this subcurve is connected. Then there exists a chain of subcurves consisting of connected components of $\nonex$ and exceptional components other than $E_{i_0}$ connecting $\nonexcomp$ and $\nonexcompi$. On the one hand $\isoii_j\neq\isoii_{j'}$, since $\tau_{i_0}\mapsto-\tau_{i_0}$. On the other hand the coordinates corresponding to the exceptional components in the chain are all fixed, hence $\isoii_j=\isoii_{j'}$. Therefore, $\overline{\spini\setminus E_{i_0}}$ has two connected components $\spini_1$ and $\spini_2$ and w.l.o.g. all $\isoii_j$ of connected components $\nonexcomp$ contained in $\spini_1$ resp. $\spini_2$ are equal to $0$ resp. $1$. This means that $\isoii$ is multiplication with $1$ resp. $-1$ in fibres of $\spinii$ over $\spini_1$ resp. $\spini_2$, which is case (i) of the proposition.
\end{proof}
\begin{rem}\label{rem:H}
The above proof also shows that in case (i) the action of $\iso$ is $\tau_{i_0}\mapsto-\tau_{i_0}$ and $\tau_j\mapsto\tau_j$, $j\neq i_0$, if $\tau_{i_0}$ is the coordinate corresponding to the exceptional component $E_{i_0}$, which devides $\spini$ into the two connected components $\spini_1$ and $\spini_2$. Moreover, in case (ii) the action of $\iso$ is $\tau_{i_0}\mapsto\xi_4\tau_{i_0}$ and $\tau_j\mapsto\tau_j$, $j\neq i_0$, where $\xi_4$ is an appropriate root of $-1$ and $E_{i_0}$ the exceptional component connecting the elliptic tail on which $\iso$ acts as elliptic involution to the rest of the curve. 

Therefore, the subgroup of $\paut\spin$ generated by quasireflections consists of diagonal matrices 
{\footnotesize\[
\begin{pmatrix}
d_1&&\\[-1,5ex]
&\!\!\!\!\!\ddots&\\[-1ex]
&&\!\!\!\!\!d_{3g-3}
\end{pmatrix}
\]}where $d_i^4=1$ if $\tau_i$ corresponds to an exceptional component meeting an elliptic tail, $d_i^2=1$ if $\tau_i$ corresponds to a disconnecting exceptional component which does not meet an elliptic tail and $d_i=1$ in all other cases, i.e. if $\tau_i$ corresponds to a non-disconnecting exceptional component or a non-exceptional node or a non-exceptional component.
\end{rem}
\begin{notation}\label{not:part}
We will use the following notation for a spin curve $\spin$ with stable model $\stab$.
\begin{align*}
\ET=\,\ET\spin&=\{P_i\in\sing\stab \text{ is an elliptic tail node}\}\\
\DC=\DC\spin&=\{P_i\in\sing\stab\setminus T\spin \text{ is a disconnecting node}\}\\
\EX=\EX\spin&=\{P_i\in\sing\stab \text{ is a non-disconnecting exceptional node}\}\\
\Delta=\Delta\spin&=\{P_i\in\sing\stab\text{ is a non-exceptional node}\}
\end{align*}
Note that the first three sets form a partition of the set $N=N\spin\subset\sing\stab$ of exceptional nodes and all four sets form a partition of $\sing\stab$. Moreover, every automorphism $\isoi_\stab$ which lifts to $\spin$ fixes this partition. 
\end{notation}
\begin{proof}[Proof of Theorem~\ref{thm:smooth}]
\emph{First step: The two conditions are sufficient.} Let $\spin$ be a spin curve fulfilling the two conditions of the theorem. We have to show that these conditions imply that $\paut\spin$ is generated by elements acting as quasireflections, i.e. that every automorphism $\iso\in\aut\spin$ can be written as a product of elements acting as quasireflections or as the identity. 

\emph{Claim: We may assume w.l.o.g. that $\iso$ is an inessential automorphism.} Let $\iso$ be any automorphism of $\spin$ and $\isoi_\stab$ the induced automorphism of the stable model, in particular $\isoi_\stab$ lifts to $\spin$. By condition (ii) $\isoi_\stab$ can be written as a composition of elliptic tail automorphisms of order $2$, i.e. there exist elliptic tails $\stab_1,\dotsc,\stab_k$ of $\stab$ such that $\isoi_\stab=\iota_1\circ\dotsb\circ\iota_k$, where $\iota_i$ is the elliptic tail automorphism of order $2$ with respect to $\stab_i$. 

Denote by $P_i$ the elliptic tail node connecting $\stab_i$ to the rest of the curve. In case $\stab_i$ is a smooth elliptic tail $\spinii_i^\nu=\nu_{\spini}^*\spinii_{|\stab^\nu_i}$ corresponds to a two-torsion point of $\stab_i^\nu$, where $\nu_{\spini}:\spini^\nu\rightarrow\spini$ is the normalisation, hence $({\iota^\nu_i}_{|\stab^\nu_i})^*\spinii_i^\nu\cong\spinii_i^\nu$. In case $\stab_i$ is a singular elliptic tail, denote by $Q$ the node contained in $\stab_i$. If the node $Q$ is exceptional $\spinii_i^\nu\cong\calO_{\PP[1]}(-1)$ and $({\iota^\nu_i}_{|\stab^\nu_i})^*\spinii_i^\nu\cong\spinii_i^\nu$. If the node $Q$ is non-exceptional $\spinii_i^\nu=\calO_{\PP[1]}$, $({\iota^\nu_i}_{|\stab^\nu_i})^*\spinii_i^\nu\cong\spinii_i^\nu$ and also $({\iota_i}_{|\stab_i})^*\spinii_{|\stab_i}\cong\spinii_{|\stab_i}$. In all cases $\iota_i$ fixes all nodes of $\stab$, ${\iota_i}_{|\overline{\stab\setminus\stab_i}}$ is the identity and we have $\widetilde{\iota_i}^*\widetilde\spinii\cong\widetilde\spinii$. Therefore, every $\iota_i$ lifts to the spin curve $\spin$. 

Abusing notation denote by $(\iota_i,\isoii^{(i)})\in\aut\spin$, $i=1,\dotsc,k$, a lift of $\iota_i\in\aut\stab$ such that $\iota_i$ is the identity on $\overline{\spini\setminus(\comp\cup E_i)}$, where $E_i$ is the exceptional component over $P_i$. Consider the concatenation $(\isoi',\isoii')=\iso\circ(\iota_1,\isoii^{(1)})\circ\dotsb\circ(\iota_k,\isoii^{(k)})$. By construction $\isoi'=\isoi\circ\iota_1\circ\dotsb\circ\iota_k$ in $\aut\spini$ is the identity on \emph{every} non-exceptional component $\comp$ of $\stab$. Therefore, $(\isoi',\isoii')$ is an inessential automorphism. If we prove that every inessential automorphism can be written as a concatenation of (inessential) automorphisms acting as quasireflections or as the identity we are done, since by Proposition~\ref{prop:qrsg} the $(\iota_i,\isoii^{(i)})$ act as quasireflections.

Now assume that $\iso$ is an inessential automorphism. If $\iso$ acts as the identity we are done, so assume that $\iso\neq(\identity_\spini,\pm\identity_\spinii)$. Let $(\isoii_j)_j\in\ZZ[V(\Sigma(\spini))]_2$ be the element corresponding to $\iso$, i.e. $\isoii$ is multiplication with $(-1)^{\isoii_j}$ in every fibre of $\spinii$ over the $j$th connected component $\nonexcomp$ of the non-exceptional subcurve $\nonex$. Recall from the proof of Proposition~\ref{prop:qrsg} that $\iso$ acts as $\tau_i\mapsto-\tau_i$, if $\tau_i$ corresponds to an exceptional component connecting two connected components $\nonexcomp$ and $\nonexcompi$ with $\isoii_j\neq\isoii_{j'}$, and $\tau_i\mapsto\tau_i$ in all other cases. By Remark~\ref{rem:H} this action is a composition of quasireflections in $\paut\spin$ if and only if $\tau_i\mapsto\tau_i$ for all coordinates $\tau_i$ corresponding to components of $\stab$, non-exceptional nodes or non-disconnecting exceptional nodes. 

We already know that $\iso$ acts trivially on coordinates corresponding to components of $\stab$ or non-exceptional nodes. Let $\tau_i$ correspond to a non-disconnecting exceptional node $P_i\in\overline{N}\spin$. This node gives a non-disconnecting edge in the graph $\Sigma(\spini)$. Since the graph $\Sigma(\spini)$ is tree-like, every non-disconnecting edge is a loop. This in turn means that for the corresponding exceptional component of $\spini$ the two connected components $\nonexcomp$ and $\nonexcompi$ coincide, in particular $\isoii_j=\isoii_{j'}$ and $\tau_i\mapsto\tau_i$. Hence $\iso$ is a concatenation of automorphisms acting as quasireflections or as the identity.

\emph{Second step: The two conditions are necessary.} Let $[\spin]\in\sgbar$ be a smooth point. Hence $\paut\spin$ is generated by elements acting as quasireflections. Let $\isoi_\stab$ be an automorphism of the stable model $\stab$ lifting to an automorphism $\iso$ of $\spin$. $\iso$  can be decomposed into a product of elements acting as quasireflections (modulo $(\identity_\spini,\pm\identity_\spinii)$). These induce either the identity or elliptic tail automorphisms of order $2$ on $\stab$. Therefore, $\isoi_\stab$ can be written as a concatenation of elliptic tail automorphisms of order $2$ and condition (ii) is necessary.

Now assume that $\Sigma(\spini)$ is not tree-like. This implies that there exists a cycle $e_1,\dotsc,e_k$ of edges in $\Sigma(\spini)$ such that no $e_i$, $i=1,\dotsc,k$, is a loop. Let $v_1,\dotsc,v_k$ be the vertices of this cycle, i.e. $e_i$ connects the vertices $v_i$ and $v_{i+1}$, where the indices are considered modulo $k$. We may assume that the $v_i$ are pairwise distinct. 

\[\begindc{1}[2]
\obj(0,10)[vk]{$v_k$}[7] 
\obj(10,20)[v1]{$v_1$}[8] 
\obj(20,20)[v2]{$v_2$}[2] 
\obj(30,10)[v3]{$v_3$}[3] 
\obj(20,0)[vi]{$v_j$}[4] 
\obj(10,0)[vi1]{$v_{j+1}$}[6] 
\mor{vk}{v1}{$e_k$}[-1,2] 
\mor{v1}{v2}{$e_1$}[-1,2] 
\mor{v2}{v3}{$e_2$}[-1,2] 
\mor{vi}{vi1}{$e_j$}[-1,2] 
\mor(0,10)(3,7){}[1,2] 
\mor(4,6)(7,3){}[1,2] 
\mor(8,2)(10,0){}[1,2] 
\mor(20,0)(22,2){}[1,2] 
\mor(23,3)(26,6){}[1,2] 
\mor(27,7)(30,10){}[1,2] 
\enddc\]

Denote by $E_i$ the exceptional component corresponding to $e_i$ and by $\nonex_i$ the connected component of $\nonex$ corresponding to $v_i$. Consider the inessential automorphism $\iso$ of $\spin$ which is multiplication with $-1$ over the component $\nonex_1$ and multiplication with $1$ over all other connected components of $\nonex$, i.e. $\isoii_1=1$ and $\isoii_j=0$ else. $\iso$ acts as $\tau_1\mapsto-\tau_1$, $\tau_k\mapsto-\tau_k$ and $\tau_j\mapsto\tau_j$ else. In particular the action is non-trivial on a coordinate corresponding to a non-disconnecting exceptional component. By Remark~\ref{rem:H} such an action is not a product of quasireflections in $\paut\spin$. Therefore, the graph $\Sigma(\spini)$ is tree-like for a smooth point $[\spin]\in\sgbar$.
\end{proof}
\begin{cor}
For $g\geq 4$ the image of the singular locus $\sing\sgbar$ under the forgetful morphism $\pi:\sgbar\rightarrow\mgbar$ is
\[
\pi(\sing\sgbar)=\sing\mgbar\cup\left\{[\stab]\in\mgbar|\Gamma(\stab) \text{ is not tree-like}\right\},
\]
where $\Gamma(\stab)$ is the dual graph of the stable curve $\stab$.
\end{cor}
\begin{proof}
``$\subset$'' For $g\geq 4$ let $[\spin]$ be a singular point of $\sgbar$. By Theorem~\ref{thm:smooth} either $\Sigma(\spini)$ is not tree-like or there exists an automorphism $\isoi_\stab\in\aut\stab$ which lifts to $\spin$ and is not a composition of elliptic tail automorphisms of order $2$. In the second case $\aut\stab$ is not generated by elliptic tail automorphisms of order $2$, hence by Remark~\ref{rem:qrmg} $[\stab]$ is a singular point of $\mgbar$. 

In the first case $\Sigma(\spini)$ is not tree-like. Let $\Delta\subset\sing\stab$ be the set of non-exceptional nodes and $N$ the set of exceptional nodes. The restriction of the stable model $\beta:\spini\rightarrow\stab$ to the non-exceptional subcurve $\nonex$ is the partial normalisation of $\stab$ at $N$. Therefore, all irreducible components of $\stab$ which are connected by nodes in $\Delta$ lie in the same connected component $\nonexcomp$ of $\nonex$ while the set of exceptional nodes $N$ is in $1:1$-correspondence to the exceptional components. Hence contracting all edges $e(P)$ in $\Gamma(\stab)$ corresponding to nodes $P\in\Delta$ gives the graph $\Sigma(\spini)$. Since $\Sigma(\spini)$ is not tree-like and comes from $\Gamma(\stab)$ by contracting a subset of edges, $\Gamma(\stab)$ cannot be tree-like either.

``$\supset$'' \emph{First step: If $\stab$ is a stable curve such that $\Gamma(\stab)$ is not tree-like, then $[\stab]\in\pi(\sing\sgbar)$.} Let $\stab$ be a stable curve of genus $g\geq 4$ such that the dual graph $\Gamma(\stab)$ is not tree-like. The set $\Delta=\emptyset\subset\sing\stab$ is an even subset. This implies that there exists a spin curve $\spin$ such that $\beta:\spini\rightarrow\stab$ is the blow up of $\stab$ at $N=\sing\stab\setminus\Delta=\sing\stab$. Since the graph $\Sigma(\spini)$ can be constructed from $\Gamma(\stab)$ by contracting all edges corresponding to non-exceptional nodes, i.e. nodes in $\Delta=\emptyset$, the graphs $\Sigma(\spini)$ and $\Gamma(\stab)$ coincide. In particular $\Sigma(\spini)$ is not tree-like, $[\spin]\in\sgbar$ is a singular point and $[\stab]=\pi([\spin])\in\pi(\sing\sgbar)$.

\emph{Second step: If $\stab$ is a smooth curve and $[\stab]\in\sing M_g$, then $[\stab]\in\pi(\sing\sgbar)$.} Let $\stab$ be a smooth curve of genus $g\geq 4$ such that $[\stab]$ is a singular point of $M_g$. Then there exists a non-trivial automorphism $\isoi_\stab\in\aut\stab$. M.~Atiyah proved in his article~\cite{at1971} that there exists a theta characteristic $\spinii$ on $\stab$ which is fixed by $\isoi_\stab$, i.e. $\isoi_\stab^*\spinii\cong\spinii$. Therefore $\isoi_\stab$ lifts to the spin curve $(\stab,\spinii,\spiniii)$ and $[(\stab,\spinii,\spiniii)]$ is a singular point of $S_g$.

\emph{Third step: If $\stab$ is a singular stable curve such that the dual graph $\Gamma(\stab)$ is tree-like and $[\stab]\in\mgbar$ is a singular point, then $[\stab]\in\pi(\sing\sgbar)$.} Let $\stab$ be such a curve. By Remark~\ref{rem:qrmg} there exists an automorphism $\isoi_\stab\in\aut\stab$ which is not a product of elliptic tail automorphisms of order $2$. If we can find a spin curve $\spin$ with stable model $\stab$ such that $\isoi_\stab$ lifts to $\spin$ the point $[\spin]\in\sgbar$ is singular and $[\stab]\in\pi(\sing\sgbar)$. 

Let $\Delta=\emptyset\subset\sing\stab$ and $\beta:\spini\rightarrow\stab$ the blow up at $N=\sing\stab\setminus\Delta=\sing\stab$. The non-exceptional subcurve $\nonex$ is the normalization of $\stab$, in particular $\nonex$ is a disjoint union $\coprod_j\compnu$ of smooth curves, where the union is taken over all irreducible components $\comp$ of $\stab$ and $\compnu\rightarrow\comp$ is the normalisation. The automorphism $\isoi_\stab$ induces an automorphism $\widetilde\isoi$ on $\nonex$. We want to define a line bundle $\widetilde\spinii$ on $\nonex$ such that $\widetilde\isoi^*\widetilde\spinii\cong\widetilde\spinii$. Consider a component $\stab^\nu_{i_0}$ of $\nonex$ and let $m$ be the smallest number such that $\stab^\nu_{i_0}, \stab^\nu_{i_1}=\widetilde\isoi(\stab^\nu_{i_0}),\dotsc,\stab^\nu_{i_{m-1}}=\widetilde\isoi^{m-1}(\stab^\nu_{i_0})$ are distinct and $\widetilde\isoi^m(\stab^\nu_{i_0})=\stab^\nu_{i_0}$. By Atiyah's result there exists a theta characteristic $\widetilde\spinii_{i_0}$ on $\stab^\nu_{i_0}$ which is fixed by $\left(\widetilde\isoi_{|\stab^\nu_{i_0}}\right)^m$. Fix an isomorphism $\spiniii_{i_0}:\osquare{\widetilde\spinii_{i_0}}\overset{\cong}{\longrightarrow}\omega_{\stab_{i_0}^\nu}$ and let $(\widetilde\spinii_{i_j},\spiniii_{i_j})$ for $j=1,\dotsc,m-1$ be the appropriate pull back of $(\widetilde\spinii_{i_0},\spiniii_{i_0})$ to $\stab^\nu_{i_j}$, i.e. 
\[
\widetilde\spinii_{i_j}=\left({\widetilde\isoi^{m-j}}{}_{|\stab_{i_j}^\nu}\right)^*\widetilde\spinii_{i_0}
\quad\text{and}\quad
\spiniii_{i_j}=\left({\widetilde\isoi^{m-j}}{}_{|\stab_{i_j}^\nu}\right)^*\spiniii_{i_0},
\]
in particular $\widetilde\spinii_{i_j}$ is a theta characteristic on $\stab^\nu_{i_j}$. Let $\widetilde\spinii$ be the line bundle on $\nonex$ which is $\widetilde\spinii_{i_j}$ on $\stab^\nu_{i_j}$ and $\widetilde\spiniii:\osquare{\widetilde\spinii}\overset{\cong}{\longrightarrow}\omega_{\nonex}$ be the isomorphism which is $\spiniii_{i_j}$ on $\stab_{i_j}^\nu$. Then by construction $\widetilde\isoi^*\widetilde\spinii\cong\widetilde\spinii$. 

Let $\spinii$ be a line bundle on $\spini$ which restricts to $\widetilde\spinii$ on $\nonex$ and to $\calO_E(1)$ on every exceptional component $E$. Moreover, extend $\widetilde\spiniii$ by $0$ on $E$ to get a homomorphism $\spiniii:\osquare\spinii\rightarrow\beta^*\omega_\stab$. This gives a spin curve $\spin$ with stable model $\stab$. By Proposition~\ref{prop:aut} the automorphism $\isoi_\stab$ lifts to $\spin$ since $\widetilde\isoi^*\widetilde\spinii\cong\widetilde\spinii$. Therefore, $[\spin]\in\sing\sgbar$ and $[\stab]\in\pi(\sing\sgbar)$.
\end{proof}
\begin{cor}\label{cor:nottreelike}
Let $\stab$ be a stable curve of genus $g\geq 4$ such that $\Gamma(\stab)$ is not tree-like. Then the fibre $\fibre$ of $\pi:\sgbar\rightarrow\mgbar$ over $[\stab]$ contains a singular point of $\sgbar$, i.e. $[\stab]\in\pi(\sing\sgbar)$. If in addition $\aut\stab$ is generated by elliptic tail automorphisms of order $2$, then $\fibre$ contains a smooth point of $\sgbar$.
\end{cor}
\begin{proof}
Let $\stab$ be a stable curve of genus $g\geq 4$ with $\Gamma(\stab)$ not tree-like. Then $\Delta=\emptyset\subset\sing\stab$ is an even subset and there exists a spin curve $\spin$ with support $\spini$, where $\beta:\spini\rightarrow\stab$ is the blow up of $\stab$ at $N=\sing\stab\setminus\Delta=\sing\stab$. Since $\Sigma(\spini)$ is the contraction of $\Delta$ considered as a subset of the set $E(\Gamma(\stab))$ of edges of $\Gamma(\stab)$, the graphs $\Sigma(\spini)$ and $\Gamma(\stab)$ coincide. In particular $\Sigma(\spini)$ is not tree-like and $[\spin]\in\sgbar$ is singular. 

In case $\aut\stab$ is generated by elliptic tail automorphisms a point $[\spin]\in\fibre$ is smooth if and only if $\Sigma(\spini)$ is tree-like. Consider the following subset $\Delta\subset E(\Gamma(\stab))$. An edge $e(P)$ belongs to $\Delta$ if and only if there exists a cycle of edges in $\Gamma(\stab)$ which contains $e(P)$ and $e(P)$ is not a loop. Then $\Delta$ is an even subset and there exists a spin curve $\spin$ whose support $\spini$ is the blow-up of $\stab$ at $N=\sing\stab\setminus\Delta$. The graph $\Sigma(\spini)$ is then obtained by contracting all edges of $\Gamma(\stab)$ contained in $\Delta$, i.e. by contracting all cycles of edges in $\Gamma(\stab)$ which are not loops. The resulting graph is tree-like and $[\spin]\in\sgbar$ is smooth. Note that in the same way any even $\Delta'\supset\Delta$ gives rise to smooth points of $\sgbar$. 
\end{proof}
\begin{rem}
The question which automorphisms $\isoi_\stab$ of $\stab$ lift to a spin curve $\spin$ with stable model $\stab$ is difficult, even in the case of a smooth curve $\stab$. By Atiyah's result for every automorphism $\isoi_\stab\in\aut\stab$ of a smooth curve $\stab$ there exists at least one theta characteristic $\spinii$ on $\stab$ such that $\isoi_\stab$ lifts to the spin curve $(\stab,\spinii,\spiniii)$. Moreover, S.~Kallel and D.~Sjerve show in~\cite{kasj2006} that an automorphism $\isoi_\stab\in\aut\stab$, where $\stab$ is smooth, lifts to \emph{every} theta characteristic $\spinii$ on $\stab$ if and only if $\stab$ is hyperelliptic and $\isoi_\stab$ is the hyperelliptic involution. Therefore, for every automorphism $\isoi_\stab$ which is not a hyerelliptic involution there exists at least one theta characteristic to which $\isoi_\stab$ lifts and at least one to which it does not lift. It seems to be an interesting question, which subgroups of $\aut\stab$ actually arise as the stabiliser of a theta characteristic on the smooth curve $\stab$. More generally, the question would be, which subgroups of the automorphism group of a stable curve $\stab$ do arise as the image of $\aut\spin\rightarrow\aut\stab$, where $\spin$ is any spin curve with stable model $\stab$.
\end{rem}

\section{Canonical singularities of $\sgbar$}
We will prove the following characterisation of the locus of non-canonical singularities of $\sgbar$.
\begin{thm}\label{thm:cansing}
Let $g\geq 4$. A point $[\spin]\in\sgbar$ is a non-canonical singularity if and only if the stable model $\stab$ of $\spini$ has a smooth elliptic tail $\comp$ with $j$-invariant $0$ and the theta characteristic $L_j^\nu=\nu_\spini^*\spinii_{|\compnu}$ on the elliptic tail is trivial, where $\nu_\spini:\spini^\nu\rightarrow\spini$ is the normalisation.
\end{thm}

For the proof we would like to use the \rsbt{} criterion. Consider a quotient singularity $\CC[m]/G$ with $G\subset\GL(\CC[m])$ finite. 
\begin{defn}
Let $M\in G$ be a non-trivial element of order $n$ and $\zeta$ any primitive $n$th root of $1$. Then $M$ is diagonalisable as
\[
M\sim{\footnotesize\begin{pmatrix}
\zeta^{a_1}&&\\[-1ex]
&\!\!\!\!\!\!\ddots&\\[-1ex]
&&\!\!\!\zeta^{a_m}\!\!
\end{pmatrix}}
\] 
with $0\leq a_j < n$. The \emph{\rsbt{} sum of $M$ with respect to $\zeta$} is then
\[
\RT[m].
\]
\end{defn}
\begin{rem}
Note that a quasireflection $M$ of order $n$ diagonalises as
\[
M\sim{\footnotesize\begin{pmatrix}
\zeta^{a_1}\!\!&&&\\[-0.5ex]
&\!\!\!\!1&&\\[-1ex]
&&\!\!\!\!\ddots&\\[-1ex]
&&&\!\!\!\!1
\end{pmatrix}}
\] 
for a primitive $n$th root $\zeta$ of $1$, hence $1\leq a_1<n$, $a_2=\dotsb=a_m=0$ and
\[
0<\RT[m]=\frac{a_1}{n}<1.
\]
\end{rem}
\begin{thm}\emph{(\rsbt{} criterion, \cite{re1980}, \cite{ta1982})} Let $G\subset\GL(\CC[m])$ be a finite group without quasireflections. $\CC[m]/G$ has canonical singularities if and only if for every non-trivial element $M\in G$ and every primitive $\ord M$-th root $\zeta$ of $1$ the \rsbt{} sum of $M$ with respect to $\zeta$ fulfils the \rsbt{} inequality
\[
\RT[m]\geq 1.
\] 
\end{thm}
Let $\spin$ be a spin curve of genus $g\geq 4$. By Proposition~\ref{prop:qrsg} $\paut\spin$ contains a quasireflection if and only if the stable model $\stab$ of $\spini$ is contained in some boundary divisor $\Delta_i$, $i=1,\dotsc,[\frac g 2]$, i.e. there exists a disconnecting node $P\in\sing\stab$ such that the partial normalisation of $\stab$ at $P$ has two connected components of genera $i$ and $g-i$ respectively. Therefore, we cannot apply the \rsbt{} criterion to the quotient $\CC[3g-3]_\tau/\paut\spin$ for arbitrary spin curves. 

We want to use the following result of D.~Prill~\cite{pr1967}.
\begin{prop}
Let $V=\CC[m]$ and $G\subset\GL(V)$ be a finite group. Then the subgroup $H\subset G$ generated by the quasireflections in $G$ is a normal subgroup of $G$, there exists an isomorphism $V/H\overset{\cong}{\longrightarrow}W=\CC[m]$ and a finite group $K\subset\GL(W)$ containing no quasireflections such that the following diagram commutes.
\[\begindc{0}[1]
\obj(0,0)[ul]{$V/G$}[1] 
\obj(0,40)[ol]{$V$}[1] 
\obj(80,0)[um]{$(V/H)/(G/H)$}[1] 
\obj(80,40)[om]{$V/H$}[1] 
\obj(160,0)[ur]{$W/K$}[1] 
\obj(160,40)[or]{$W$}[1] 
\mor{ol}{ul}[10,10]{$$}[1,0] 
\mor{om}{um}[10,10]{$$}[1,0] 
\mor{or}{ur}[10,10]{$$}[1,0] 
\mor{ol}{om}[10,10]{$$}[1,0] 
\mor{om}{or}[10,10]{\footnotesize$\cong$}[1,0] 
\mor{ul}{um}[10,10]{\footnotesize$\cong$}[1,0] 
\mor{um}{ur}[10,10]{\footnotesize$\cong$}[1,0] 
\enddc\]
\end{prop}
In our case $V=\CC[3g-3]_\tau$ and $H$ is generated by the quasireflections described in Proposition~\ref{prop:qrsg}. Therefore, the isomorphism $\CC[3g-3]_\tau/H\overset{\cong}{\longrightarrow}\CC[3g-3]_u$ is induced by the coordinate change 
\[
u_j=\begin{cases}
\tau_j^4 & P_j \in\ET=\ET\spin\\
\tau_j^2 & P_j \in\DC=\DC\spin\\
\tau_j & \text{else}
\end{cases}
\] 
Let $\iso$ be an automorphism of $\spin$ and $\isoi_\stab$ the induced automorphism of the stable model. Then $\isoi_\stab$ fixes the set $\ET$ of elliptic tail nodes, hence if $P_{j}$ is an elliptic tail node then also $P_{j'}=\isoi_\stab(P_j)\in\ET$. There exist nonzero scalars $c_j$ and $\bar{c}_j$ with $c_j^2=\bar{c}_j$ such that $t_j\mapsto\bar{c}_j t_{j'}$, $\tau_j\mapsto c_j \tau_{j'}$ and hence $u_j=\tau_j^4\mapsto c_j^4\tau_j^4=\bar{c}_j^2u_j$. If $P_j\in\DC$ its image $P_{j'}=\isoi_\stab(P_j)$ also lies in $\DC$, $t_j\mapsto\bar{c}_j t_{j'}$, $\tau_j\mapsto c_j \tau_{j'}$ where $c_j^2=\bar{c}_j$ and $u_j=\tau_j^2\mapsto c_j^2\tau_j^2=\bar{c}_ju_j$. If $P_j\in\EX$, i.e., $P_j$ is an exceptional non-disconnecting node, its image $P_{j'}=\isoi_\stab(P_j)$ also lies in $\EX$, $t_j\mapsto\bar{c}_j t_{j'}$, $\tau_j\mapsto c_j \tau_{j'}$ where $c_j^2=\bar{c}_j$ and $u_j=\tau_j\mapsto c_j\tau_j=c_ju_j$. If the coordinate $t_j$ corresponds to a non-exceptional node or a component we have $t_j=\tau_j=u_j$ and the action of $\isoi_\stab$ on $t_j$ and the actions of $\iso$ on $\tau_j$ and $u_j$ coincide. By Prill's result no element of $\aut\spin$ acts as a quasireflection on $\CC[3g-3]_u$ and  the \rsbt{} criterion applies to the quotient $\CC[3g-3]_u/\aut\spin$.
\begin{proof}[Proof of the if-part of Theorem~\ref{thm:cansing}]
Let $\spin$ be a spin curve of genus $g\geq 4$ whose stable model $\stab$ has a smooth elliptic tail $\stab_1$ with $j$-invariant $0$ and the theta characteristic $L_1^\nu=\nu_\spini^*\spinii_{|\stab_1^\nu}$ on the elliptic tail is trivial, where $\nu_\spini:\spini^\nu\rightarrow\spini$ is the normalisation. We have to prove, that there exists an automorphism $\iso$ of $\spin$ whose action on $\CC[3g-3]_u$ is non-trivial and has \rsbt{} sum less than $1$ for some appropriate root $\zeta$ of $1$.

Denote by $\isoi_\stab\in\aut\stab$ one of the two elliptic tail automorphisms of order $3$ with respect to $\stab_1$. It follows from the proof of Theorem~2 in~\cite{hamu1982} (see page 40) that $\isoi_\stab$ acts as $t_1\mapsto\zeta_3 t_1$, $t_2\mapsto\zeta_3^2 t_2$, $t_i\mapsto t_i$ else on $\CC[3g-3]_t$ for an appropriate primitive third root $\zeta_3$ of $1$, where $t_1$ is the coordinate corresponding to the elliptic tail node $P_1$ connecting $\stab_1$ to the rest of the curve and $t_2$ corresponds to the elliptic tail $\stab_1$. Since the theta characteristic $\spinii_1^\nu$ is trivial the automorphism ${\isoi_\stab^\nu}_{|\stab^\nu_1}$ fixes $\spinii_1^\nu=\calO_{\stab_1^\nu}$. Moreover, $\isoi_\stab$ is trivial on every other component $\comp$, hence $\widetilde\isoi^*\widetilde\spinii\cong\widetilde\spinii$, where $\widetilde\isoi:\nonex\overset{\cong}{\longrightarrow}\nonex$ is the unique lift of $\isoi_\stab$ to the non-exceptional subcurve $\nonex$ and $\widetilde\spinii=\spinii_{|\nonex}$. Proposition~\ref{prop:aut} then implies that $\isoi_\stab$ lifts to the spin curve $\spin$. 

Let $\isoii_1:\big({\isoi^\nu_\stab}_{|\stab_1^\nu}\big)^*\calO_{\stab_1^\nu}\rightarrow\calO_{\stab_1^\nu}$ be one of the two isomorphisms compatible with the isomorphism $\spiniii_1:\osquare{\spinii_1^\nu}=\calO_{\stab_1^\nu}^{\otimes 2}\rightarrow\omega_{\stab_1^\nu}=\calO_{\stab_1^\nu}$, i.e. such that
\[\begindc{0}[1]
\obj(0,0)[ul]{$\big({\isoi^\nu_\stab}_{|\stab_1^\nu}\big)^*\omega_{\stab_1^\nu}$}[1] 
\obj(100,0)[ur]{$\omega_{\stab_1^\nu}$}[1] 
\obj(0,45)[ol]{$\big({\isoi^\nu_\stab}_{|\stab_1^\nu}\big)^*\calO_{\stab_1^\nu}^{\otimes 2}$}[1] 
\obj(100,45)[or]{$\calO_{\stab_1^\nu}^{\otimes 2}$}[1] 
\mor{ul}{ur}[10,10]{\footnotesize$$}[1,0] 
\mor{ol}{or}[10,10]{\footnotesize$\isoii_{\stab_1^\nu}^{\otimes 2}$}[1,0] 
\mor{ol}{ul}[10,10]{\footnotesize$\big({\isoi^\nu_\stab}_{|\stab_1^\nu}\big)^*\spiniii_1$}[1,0] 
\mor{or}{ur}[10,10]{\footnotesize$\spiniii_1$}[1,0] 
\enddc\]
commutes. For every connected component $\nonexcomp\neq\stab_1$ the restriction $\widetilde\isoi_{|\nonexcomp}$ is the identity. Let $\isoii_j:\big(\widetilde\isoi_{|\nonexcomp}\big)^*\spinii_{|\nonexcomp}=\spinii_{|\nonexcomp}\rightarrow\spinii_{|\nonexcomp}$ be the identity. The isomorphisms $\isoii_j$ and $\isoii_1$ give an isomorphism $\widetilde\isoii:\widetilde\isoi^*\widetilde\spinii\rightarrow\widetilde\spinii$ compatible with $\widetilde\spiniii:\osquare{\widetilde\spinii}\rightarrow\omega_{\nonex}$. Hence by Proposition~\ref{prop:aut} there is a unique automorphism $\iso$ extending $(\widetilde\isoi,\widetilde\isoii)$.

The action of $\iso\in\aut\spin$ on $\CC[3g-3]_\tau$ is then the following. $\tau_1\mapsto\zeta_6\tau_1$, where $\zeta_6$ is an appropriate root of $\zeta_3$ (depending on the choice of $\isoii_1$), since $t_i\mapsto\zeta_3 t_i$ and $\tau_1^2=t_1$. For all other exceptional nodes $t_i\mapsto t_i$ and if $E_i$ is the corresponding exceptional component with incident connected components $\nonexcomp$ and $\nonexcompi$ of $\nonex$ then $\isoii$ is the identity on fibres of $\spinii$ over these components, hence $\tau_i\mapsto\tau_i$. For all other coordinates the action of $\iso$ on $\tau_i$ is the same as that of $\isoi_\stab$ on $t_i$. In particular $\tau_2\mapsto\zeta_3^2\tau_2$ and $\tau_i\mapsto\tau_i$ else. For the $u$-coordinates this implies 
\begin{align*}
u_1=\tau_1^4=t_1^2&\mapsto\zeta_3^2 u_1\\
u_2=\tau_2=t_2&\mapsto\zeta_3^2 u_2\\
u_i&\mapsto u_i \qquad\text{ else}
\end{align*}
The \rsbt{} sum of the action of $\iso$ on $\CC[3g-3]_u$ with respect to the primitive third root $\zeta_3^2$ of $1$ is then
\[
\frac 13(1+1+0+\dotsb+0)=\frac 23<1
\]
and by the \rsbt{} criterion $\CC[3g-3]_u/\aut\spin$ has a non-canonical singularity at $0$. Therefore, $[\spin]\in\sgbar$ is a non-canonical singularity. 
\end{proof}
We will now prove the ``only if'' part of Theorem~\ref{thm:cansing}. Let $[\spin]\in\sgbar$ be a non-canonical singularity. The \rsbt{} criterion yields the existence of an automorphism $\iso\in\aut\spin$ which acts non-trivially on $\CC[3g-3]_u$, say the order of this action is $n\geq 2$, and a primitive $n$th root $\zeta$ of $1$ such that the \rsbt{} sum of the action with respect to $\zeta$ satisfies
\[
0<\RT<1,
\]
where the action has eigenvalues $\zeta^{a_j}$ with $0\leq a_j<n$. We have to show that this implies that the stable model $\stab$ of $\spini$ has an elliptic tail $\comp$ with $j$-invariant $0$  such that $\spinii_{|\comp}=\calO_{\comp}$.

In a first step we may pass to a ``more general'' non-canonical singularity without loss of generality. Let $P_{i_0}\in\sing\stab$ be a non-disconnecting node of $\stab$, i.e. $P_{i_0}\in\EX\spin\cup\Delta\spin$, and let $P_{i_1}=\isoi_\stab(P_{i_0}),\dotsc,P_{i_{m-1}}=\isoi_\stab^{m-1}(P_{i_0})$ be distinct while $\isoi_\stab^m(P_{i_0})=P_{i_0}$. Then $P_{i_j}\in\EX\cup\Delta$ and the action of $\iso$ on $\bigoplus_{j}\CC_{\tau_{i_j}}$ is
\[
\tau_{i_j}\mapsto c_{i_j}\tau_{i_{j+1}}
\]
for appropriate non-zero scalars $c_{i_j}$, where the index $j$ is considered modulo $m$. This action has characteristic polynomial $x^m-\prod_jc_{i_j}$ and its eigenvalues are the $m$th roots of $\prod_jc_{i_j}$.
\begin{prop}\label{prop:singred}
We may assume wlog that the pair $(\spin,\iso)$ is \emph{singularity reduced}, i.e. $\prod_jc_{i_j}\neq 1$ for every cycle of non-disconnecting nodes of $\stab$ as above.
\end{prop}
\begin{proof}
Let $P_{i_0},\dotsc,P_{i_{m-1}}$ be a cycle of non-disconnecting nodes as above such that $\prod_jc_{i_j}=1$. The idea is to deform the spin curve $\spin$ to a nearby spin curve $(\spini',\spinii',\spiniii')$ in such a way that the nodes $P_{i_j}$ are smoothed and the automorphism $\iso$ deforms to an automorphism $(\isoi',\isoii')$ of $(\spini',\spinii',\spiniii')$. We then have to prove that the actions of $\iso$ on $\CC[3g-3]_u$ and the respective action of $(\isoi',\isoii')$ have the same eigenvalues, hence the same \rsbt{} sum, and that $\spin$ has an elliptic tail with $j$-invariant $0$ and trivial theta characteristic if and only if $(\spini',\spinii',\spiniii')$ does.

The $m$th power of the action of $\iso$ on $W=\bigoplus_j\CC_{\tau_{i_j}}$ is given by
\[
\tau_{i_j}\mapsto \Big(\prod_kc_{i_k}\Big)\tau_{i_j},\qquad j=1,\dotsc,m.
\]
Hence the assumption $\prod_kc_{i_k}=1$ implies that this action is trivial. Let $w_0\in\CC_{\tau_{i_0}}$ be a nonzero element and set $w=\sum_{j=0}^{m-1}\iso^j w_0\in W$. Let $(\spini',\spinii',\spiniii')$ be the fibre of the local universal deformation $(\calX\rightarrow\CC[3g-3]_\tau,\cspinii,\cspiniii)$ of $\spin$ over $w$. Since $\iso^m$ is trivial on $W$, the element $w$ is fixed by $\iso$
\[
\iso w=\sum_{j=0}^{m-1}\iso^{j+1}w_0=\iso^mw_0+\sum_{j=1}^{m-1}\iso^jw_0=w
\]
which means that $\iso$ deforms to an automorphism $(\isoi',\isoii')$ of $(\spini',\spinii',\spiniii')$.
Moreover, every summand $\iso^jw_0\in\CC_{\tau_{i_j}}$ is non-zero. Hence the node $P_{i_j}$ is smoothed in $(\spini',\spinii',\spiniii')$.

Applying this argument to every cycle of non-disconnecting nodes of $\stab$ with $\prod_jc_{i_j}=1$ gives a spin curve $(\spini',\spinii',\spiniii')$ where all these nodes are smoothed and $\iso$ deforms to $(\isoi',\isoii')\in\aut(\spini',\spinii',\spiniii')$. The pair $((\spini',\spinii',\spiniii'),(\isoi',\isoii'))$ is then singularity reduced, i.e. for all cycles of non-disconnecting nodes $\prod_jc_{i_j}\neq 1$. This deformation does not affect the disconnecting nodes of $\stab$, in particular $\spini$ and $\spini'$ have the same elliptic tail nodes and elliptic tails. On the one hand this means 
that $\spin$ has an elliptic tail with $j$-invariant $0$ with trivial theta characteristic if and only if $(\spini',\spinii',\spiniii')$ does. On the other hand this also implies that the subgroups of $\paut\spin$ and $\paut(\spini',\spinii',\spiniii')$ generated by quasireflections coincide. 

Let us compare eigenvalues of the actions of $\iso$ and $(\isoi',\isoii')$. The actions of the two automorphisms on the $\tau$-coordinates corresponding to disconnecting nodes of $\stab$ are the same. The fact that the subgroups generated by quasireflections coincide implies that the actions on the corresponding $u$-coordinates agree also. For the remaining coordinates we have $\tau_i=u_i$. The action of $\iso$ on these coordinates deforms continously to the action of the deformed automorphism $(\isoi',\isoii')$. Therefore, the eigenvalues vary continously and since every eigenvalue is an $n$th root of $1$ and these form a discrete set the eigenvalues of $\iso$ and $(\isoi',\isoii')$ on the remaining coordinates agree. In particular $\iso$ and $(\isoi',\isoii')$ have the same \rsbt{} sum.
\end{proof}
From now on we fix a singularity reduced pair $(\spin,\iso)$ and a primitive $n$th root $\zeta$ of $1$ such that the \rsbt{} sum of $\iso$ with respect to $\zeta$ satisfies
\[
0<\RT<1.
\]
\begin{prop}\label{prop:pairnodes} \emph{(compare~\cite[p. 34]{hamu1982})}
The induced automorphism $\isoi_\stab$ of the stable model $\stab$ either fixes every node but two which are interchanged or fixes every node.
\end{prop}
\begin{proof}
Let $P_{i_0},P_{i_1},\dotsc,P_{i_{m-1}}$ be a cycle of nodes of $\stab$ and denote by $W=\bigoplus_j\CC_{u_{i_j}}$ the corresponding subspace of $\CC[3g-3]_u$. The action of $\iso$ on $W$ is given by
\[
B={\footnotesize\begin{pmatrix}
0&\!\!\!\!\alpha_1&&\\[-1ex]
\vdots&&\!\!\!\!\ddots&\\[-0.5ex]
0&&&\!\!\!\!\alpha_{m-1}\\[-0.5ex]
\alpha_m&\!\!\!\!0&\!\!\!\!\dotsb&\!\!\!\!0
\end{pmatrix}}
\]
for appropriate nonzero scalars $\alpha_j$. We have already seen, that in such a situation $B^m=\left(\prod_j\alpha_{j}\right)\cdot\1$, where $\1$ is the indentity matrix. Since $n$ is the order of the action of $\iso$ on $\CC[3g-3]_u$, $m$ divides $n$ and 
\[
\1=B^n=\Big(\prod_j\alpha_{j}\Big)^{\frac nm}\cdot\1.
\]
Hence $\prod_j\alpha_{j}$ is an $\frac nm$th root of $1$, say $\zeta^{lm}$ for an appropriate $0\leq l<\frac nm$. The characteristic polynomial of $B$ is $x^m-\prod_j\alpha_j=x^m-\zeta^{lm}$. Therefore, the eigenvalues of $B$ are $\zeta^{l+j\frac nm}$ for $j=0,\dotsc,m-1$ and the corresponding part of the \rsbt{} sum is
\[
\frac 1n\sum_{j=0}^{m-1}\left(l+j\frac nm\right)=\frac{ml}{n}+\frac{m-1}{2}.
\]
This gives
\[
1>\RT\geq\frac{ml}{n}+\frac{m-1}{2}\geq\frac{m-1}{2}
\]
and $m$ is either $1$ or $2$. Suppose there are two different cycles of nodes of length $m=2$. Then every cycle contributes at least $\frac 12$ to the \rsbt{} sum, which gives the contradiction $1>\RT\geq\frac 12+\frac 12=1$. Therefore, either every node is fixed by $\isoi_\stab$ or there exists one pair of nodes which are interchanged by $\isoi_\stab$ and all other nodes are fixed.
\end{proof}
\begin{prop}
Every irreducible component $\comp$ of $\stab$ is fixed by $\isoi_\stab$.
\end{prop}
\begin{proof}
Let $\stab_{i_0}$ be an irreducible component of $\stab$, $\stab_{i_0},\stab_{i_1}=\isoi_\stab(\stab_{i_0}),\dotsc,\stab_{i_{m-1}}=\isoi^{m-1}_\stab(\stab_{i_0})$ distinct components and $\isoi^m_\stab(\stab_{i_0})=\stab_{i_0}$. Assume that $\stab_{i_0}$ is not fixed by $\isoi_\stab$, i.e. $m\geq 2$. Consider the subspace $W$ of $\CC[3g-3]_t$ corresponding to deformations of the components $\stab_{i_j}$, i.e. $W=\bigoplus_jH^1(\stab_{i_j}^\nu,T_{\stab_{i_j}^\nu}(D_{i_j}))$, where $\stab_{i_j}^\nu$ is the normalisation and $D_{i_j}$ is the set of preimages $P^\pm$ of nodes lying in $\stab_{i_j}^\nu$ considered as a divisor. If $t_i$ is a coordinate of $W$ then $t_i=\tau_i=u_i$ and the actions of $\isoi_\stab$ and $\iso$ on these coordinates coincide. 

Therefore the calculations on page 35 in~\cite{hamu1982} apply to our case. That means, if the eigenvalues of $\iso^m$ on $H^1(\stab_{i_0}^\nu,T_{\stab_{i_0}^\nu}(D_{i_0}))$ are the $\frac nm$th roots $\zeta^{ml_1},\dotsc,\zeta^{ml_d}$ of $1$ where $d=3g(\stab_{i_0}^\nu)-3+\deg D_{i_0}=\dim H^1(\stab_{i_0}^\nu,T_{\stab_{i_0}^\nu}(D_{i_0}))$ and $0\leq l_i<\frac nm$, then the eigenvalues of $\iso$ on $W$ are
\[
\zeta^{l_i+j\frac nm}, \qquad i=1,\dotsc,d,\ j=0,\dotsc,m-1.
\]
This gives
\[
1>\RT\geq\frac 1n\sum_{i=1}^d\sum_{j=0}^{m-1}\left(l_i+j\frac nm\right)=\frac{d(m-1)}{2}+\frac mn\sum_{i=1}^dl_i\geq\frac{d(m-1)}{2}
\]
and either $d=0$ with arbitrary $m\geq 2$ or $d=1$ and $m=2$. There are six possibilities for $\stab_{i_0}\subset\stab$:
\begin{enumerate}
\item $\stab_{i_0}^\nu$ elliptic, $1$ marked point, $\stab_{i_0}$ is a smooth elliptic tail

{\scriptsize
\[\begindc{0}[1]
\obj(-25,0)[smoothnu]{{\includegraphics{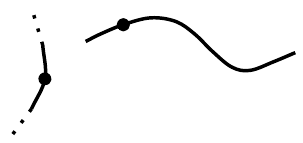}}}[0] 
\obj(150,0)[smooth]{{\includegraphics{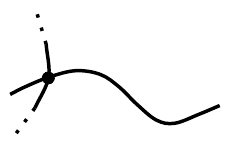}}}[0] 
\obj(125,3){$P_1$}
\obj(-31,7){$P_1^+$}
\obj(-64,0){$P_1^-$}
\obj(39,8){$\stab_{i_0}^\nu$ elliptic}
\obj(189,-8){$\stab_{i_0}$}
\mor{smoothnu}{smooth}[85,50]{$\nu_\stab$} 
\enddc\]}
and $d=1$, $m=2$. 

\item $\stab_{i_0}^\nu$ rational, $4$ marked points mapping to $2$ irreducible nodes 

{\scriptsize
\[\begindc{0}[1]
\obj(0,0)[stabnu]{{\includegraphics{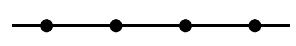}}}[0] 
\obj(150,0)[stab]{{\includegraphics{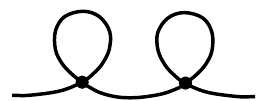}}}[0] 
\obj(137,-17){$P_1$}
\obj(169,-17){$P_2$}
\obj(-28,-8){$P_1^+$}
\obj(-8,-8){$P_1^-$}
\obj(13,-8){$P_2^+$}
\obj(33,-8){$P_2^-$}
\obj(-63,-1){rational $\stab_{i_0}^\nu$}
\obj(195,-14){$\stab_{i_0}$}
\mor{stabnu}{stab}[60,50]{$\nu_\stab$} 
\enddc\]}
and $d=1$, $m=2$.

\item $\stab_{i_0}^\nu$ rational, $4$ marked points mapping to $1$ irreducible node and $2$ non-irreducible nodes

{\scriptsize
\[\begindc{0}[1]
\obj(-60,0)[stabnu]{{\includegraphics{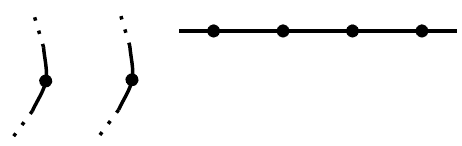}}}[0] 
\obj(150,0)[stab]{{\includegraphics{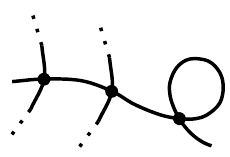}}}[0] 
\obj(136,5){$P_1$}
\obj(156,0){$P_2$}
\obj(165,-17){$P_3$}
\obj(-63,5){$P_1^+$}
\obj(-43,5){$P_2^+$}
\obj(-23,5){$P_3^+$}
\obj(-3,5){$P_3^-$}
\obj(-121,0){$P_1^-$}
\obj(-96,0){$P_2^-$}
\obj(30,11){$\stab_{i_0}^\nu$ rational}
\obj(185,-21){$\stab_{i_0}$}
\mor{stabnu}{stab}[113,60]{$\nu_\stab$} 
\enddc\]}
and $d=1$, $m=2$.
\item $\stab_{i_0}^\nu$ rational, $4$ marked points mapping to $4$ non-irreducible nodes 

{\scriptsize
\[\begindc{0}[1]
\obj(-185,0)[]{$$}[1] 
\obj(-50,0)[stabnu]{{\includegraphics{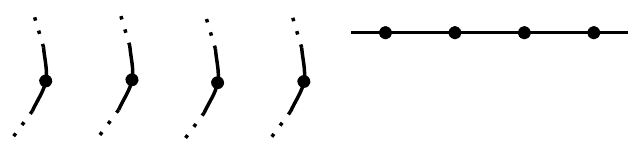}}}[0] 
\obj(175,0)[stab]{{\includegraphics{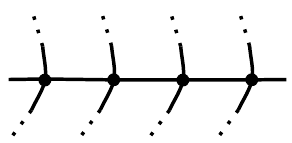}}}[0] 
\obj(150,-8){$P_1$}
\obj(170,-8){$P_2$}
\obj(190,-8){$P_3$}
\obj(210,-8){$P_4$}
\obj(-29,5){$P_1^+$}
\obj(-9,5){$P_2^+$}
\obj(11,5){$P_3^+$}
\obj(31,5){$P_4^+$}
\obj(-61,0){$P_4^-$}
\obj(-86,0){$P_3^-$}
\obj(-111,0){$P_2^-$}
\obj(-136,0){$P_1^-$}
\obj(65,11){$\stab_{i_0}^\nu$ rational}
\obj(223,-2){$\stab_{i_0}$}
\mor{stabnu}{stab}[140,50]{$\nu_\stab$} 
\enddc\]}
and $d=1$, $m=2$.
\item $\stab_{i_0}^\nu$ rational, $3$ marked points mapping to $1$ irreducible and $1$ disconnecting node, i.e.~$\stab_{i_0}$ is a singular elliptic tail 

{\scriptsize
\[\begindc{0}[1]
\obj(-20,0)[singularnu]{{\includegraphics{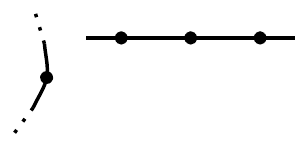}}}[0] 
\obj(150,0)[singular]{{\includegraphics{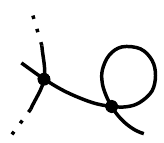}}}[0] 
\obj(132,-4){$P_1$}
\obj(155,-15){$P_2$}
\obj(-26,2){$P_1^+$}
\obj(-7,2){$P_2^+$}
\obj(14,2){$P_2^-$}
\obj(-56,0){$P_1^-$}
\obj(46,9){$\stab_{i_0}^\nu$ rational}
\obj(175,-20){$\stab_{i_0}$}
\mor{singularnu}{singular}[95,35]{$\nu_\stab$} 
\enddc\]}
and $d=0$, $m\geq 2$.
\item $\stab_{i_0}^\nu$ rational, $3$ marked points mapping to $3$ non-irreducible nodes

{\scriptsize
\[\begindc{0}[1]
\obj(-125,0)[stabnu]{{\includegraphics{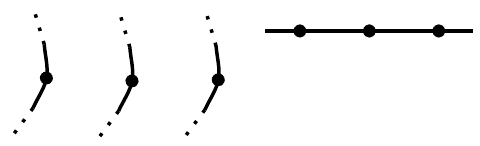}}}[0] 
\obj(90,0)[stab]{{\includegraphics{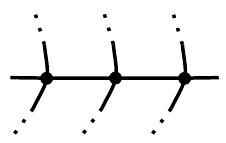}}}[0] 
\obj(76,-8){$P_1$}
\obj(96,-8){$P_2$}
\obj(116,-8){$P_3$}
\obj(-105,5){$P_1^+$}
\obj(-85,5){$P_2^+$}
\obj(-65,5){$P_3^+$}
\obj(-138,0){$P_3^-$}
\obj(-163,0){$P_2^-$}
\obj(-188,0){$P_1^-$}
\obj(-33,11){$\stab_{i_0}^\nu$ rational}
\obj(129,-2){$\stab_{i_0}$}
\mor{stabnu}{stab}[120,50]{$\nu_\stab$} 
\enddc\]}
and $d=0$, $m\geq 2$.
\end{enumerate}
Most of these cases can be excluded as in~\cite{hamu1982}. For example in case (ii) the curve $\stab$ has only genus $2$, but we assumed the genus $g$ to be at least $4$. The cases (i), (iii) and (iv) are the cases (e), (d) and (c) of~\cite{hamu1982} and give either curves of genus at most $3$ or they lead to a \rsbt{} sum bigger than $1$. In case (v) the node $P_2$ cannot be fixed. Hence by Proposition~\ref{prop:pairnodes} $P_2$ and $\isoi_\stab(P_2)\neq P_1$ are interchanged and $P_1$ is fixed. But then $m=2$, the image $\isoi_\stab(\stab_{i_0})$ is the second component through $P_1$ and $\stab$ has only genus $2$. 

Therefore, we are left with case (vi). By Proposition~\ref{prop:pairnodes} at least one of the three nodes on $\stab_{i_0}$ is fixed, say $P_1$. If all three were fixed, $\isoi_\stab(\stab_{i_0})$ would have to be the second component through all three nodes, giving a curve of genus $2$:

{\scriptsize
\[\begindc{0}[1]
\obj(0,0)[stabnu]{{\includegraphics{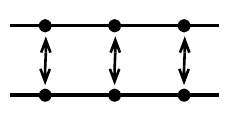}}}[0] 
\obj(150,0)[stab]{{\includegraphics{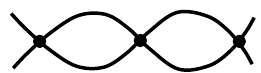}}}[0] 
\obj(125,-9){$P_1$}
\obj(153,-9){$P_2$}
\obj(181,-9){$P_3$}
\obj(-18,18){$P_1^+$}
\obj(2,18){$P_2^+$}
\obj(22,18){$P_3^+$}
\obj(22,-18){$P_3^-$}
\obj(-18,-18){$P_1^-$}
\obj(2,-18){$P_2^-$}
\obj(-54,9){rational $\stab_{i_0}^\nu$}
\obj(-63,-11){rational $\isoi_\stab^\nu(\stab_{i_0}^\nu)$}
\obj(193,7){$\stab_{i_0}$}
\obj(203,-10){$\isoi_\stab(\stab_{i_0})$}
\mor{stabnu}{stab}[55,60]{$\nu_\stab$} 
\enddc\]}

Now assume that only $P_1$ is  fixed, then the other two nodes must be interchanged, i.e.~$P_2\mapsto P_3\mapsto P_2$. Then $\isoi_\stab(\stab_{i_0})$ must be the second component through all three nodes. Again $\stab$ has only genus $2$. 
Hence exactly two of the nodes on $\stab_{i_0}$, say $P_1$ and $P_2$, are fixed and $P_3\mapsto P_4\mapsto P_3$ where $P_3\neq P_4\in\isoi_\stab(\stab_{i_0})$, giving a contribution to the \rsbt{} sum of at least $\frac 12$ by Proposition~\ref{prop:pairnodes}. $\isoi_\stab(\stab_{i_0})$ is the second component through $P_1$ and $P_2$.

{\scriptsize
\[\begindc{0}[1]
\obj(-25,0)[stabnu]{{\includegraphics{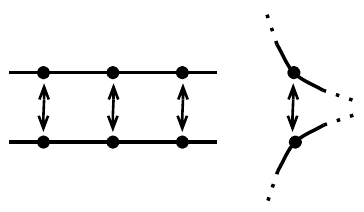}}}[0] 
\obj(150,0)[stab]{{\includegraphics{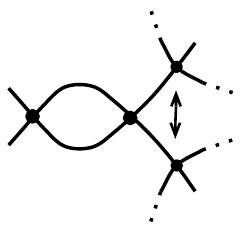}}}[0] 
\obj(117,0){$P_1$}
\obj(144,0){$P_2$}
\obj(159,15){$P_3$}
\obj(159,-14){$P_4$}
\obj(-63,18){$P_1^+$}
\obj(-43,18){$P_2^+$}
\obj(-23,18){$P_3^+$}
\obj(-23,-18){$P_4^+$}
\obj(-63,-18){$P_1^-$}
\obj(-43,-18){$P_2^-$}
\obj(15,15){$P_3^-$}
\obj(15,-15){$P_4^-$}
\obj(-5,35){$\stab_{i'_0}^\nu$}
\obj(-16,-36){$\isoi_\stab^\nu(\stab_{i'_0}^\nu)$}
\obj(-100,10){rational $\stab_{i_0}^\nu$}
\obj(-109,-10){rational $\isoi_\stab^\nu(\stab_{i_0}^\nu)$}
\obj(112,13){$\stab_{i_0}^\nu$}
\obj(102,-13){$\isoi_\stab(\stab_{i_0}^\nu)$}
\obj(194,6){$\stab_{i'_0}$}
\obj(203,-7){$\isoi_\stab(\stab_{i'_0})$}
\mor{stabnu}{stab}[75,60]{$\nu_\stab$} 
\enddc\]}

In case $\stab_{i'_0}\neq\isoi_\stab(\stab_{i'_0})$ the component $\stab_{i'_0}$ has to be as in (vi). The remaining nodes must be fixed, hence the curve has only genus $3$. 

{\scriptsize
\[\begindc{0}[1]
\obj(-50,0)[stabnu]{{\includegraphics{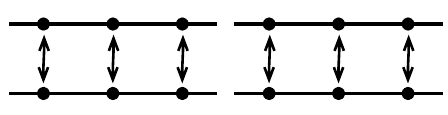}}}[0] 
\obj(165,0)[stab]{{\includegraphics{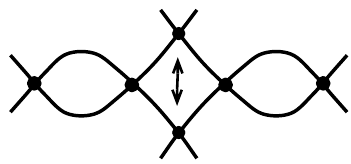}}}[0] 
\obj(116,0){$P_1$}
\obj(144,0){$P_2$}
\obj(158,14){$P_3$}
\obj(158,-14){$P_4$}
\obj(189,-1){$P_5$}
\obj(215,-1){$P_6$}
\obj(-100,18){$P_1^+$}
\obj(-80,18){$P_2^+$}
\obj(-60,18){$P_3^+$}
\obj(-60,-18){$P_4^+$}
\obj(-100,-18){$P_1^-$}
\obj(-80,-18){$P_2^-$}
\obj(-35,18){$P_3^-$}
\obj(-35,-18){$P_4^-$}
\obj(-15,18){$P_5^+$}
\obj(-15,-18){$P_5^-$}
\obj(5,18){$P_6^+$}
\obj(5,-18){$P_6^-$}
\obj(31,8){$\stab_{i'_0}^\nu \text{ rat.}$}
\obj(39,-11){$\isoi_\stab^\nu(\stab_{i'_0}^\nu)\text{ rat.}$}
\obj(-127,8){$\text{rat. }\stab_{i_0}^\nu$}
\obj(-137,-11){$\text{rat. }\isoi_\stab^\nu(\stab_{i_0}^\nu)$}
\obj(112,13){$\stab_{i_0}$}
\obj(102,-13){$\isoi_\stab(\stab_{i_0})$}
\obj(222,13){$\stab_{i'_0}$}
\obj(232,-13){$\isoi_\stab(\stab_{i'_0})$}
\mor{stabnu}{stab}[110,72]{$\nu_\stab$} 
\enddc\]}

Therefore, $\stab_{i'_0}$ is fixed and it is the second branch through $P_3$ and $P_4$. Consider the restriction $\varphi={\isoi_\stab}_{|\stab_{i_0}\cup\isoi_\stab(\stab_{i_0})}$, then $\varphi^2$ fixes the two components, all the nodes and all the marked points in the pointed normalisations. Hence $\varphi^2$ is the identity.

{\scriptsize
\[\begindc{0}[1]
\obj(0,120)[stabnu]{{\includegraphics{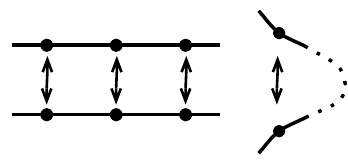}}}[0] 
\obj(-35,138){$P_1^+$}
\obj(-15,138){$P_2^+$}
\obj(5,138){$P_3^+$}
\obj(35,140){$P_3^-$}
\obj(-35,104){$P_1^-$}
\obj(-15,104){$P_2^-$}
\obj(5,104){$P_4^+$}
\obj(35,100){$P_4^-$}
\obj(-62,130){$\text{rat. }\comp^\nu$}
\obj(-72,110){$\text{rat. }\isoi_\stab^\nu(\comp^\nu)$}
\obj(60,120){$\compi^\nu$}
\obj(200,120)[stab]{{\includegraphics{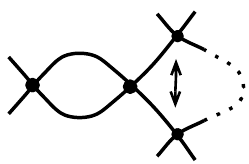}}}[0] 
\obj(166,120){$P_1$}
\obj(194,120){$P_2$}
\obj(214,143){$P_3$}
\obj(214,97){$P_4$}
\obj(160,130){$\comp$}
\obj(150,110){$\isoi_\stab(\comp)$}
\obj(245,120){$\compi$}
\mor{stabnu}{stab}[80,70]{$\nu_\stab$} 
\enddc\]}

Let $xy=0$ be a local equation of $C$ at the node $P_1$. Then $\isoi_\stab$ acts as $x\mapsto y\mapsto x$, $xy=t_1$ is the deformation of the node and $t_1=xy\mapsto yx=t_1$. The node $P_1$ is non-disconnecting and since $(\spin,\iso)$ is singularity reduced the action of $\iso$ \emph{cannot} be $\tau_1\mapsto\tau_1$. If $P_1$ were non-exceptional $\iso$ would act as $\tau_1=t_1\mapsto\tau_1=t_1$. Hence $P_1$ is exceptional, $\tau_1^2=t_1$ and $\tau_1\mapsto\pm\tau_1$. We must have $\tau_1\mapsto-\tau_1=\zeta^{\frac n2}\tau_1$, giving the contradiction
\[
1>\RT\geq \underbrace{\frac 1n \cdot\frac n2}_{P_1}+\underbrace{\frac 12}_{P_3\leftrightarrow P_4}= 1.
\]
Therefore, all cases in which an irreducible component $\stab_{i_0}$ of $\stab$ is not fixed by $\isoi_\stab$ are excluded.
\end{proof}
The proposition implies that for every irreducible component $\comp$ of $\stab$ the action of the automorphism $\isoi_\stab$ on $\CC[3g-3]_t$ restricts to an action on $H^1(\compnu,T_{\compnu}(D_j))$ and this action coincides with that of $\iso$ on $H^1(\compnu,T_{\compnu}(D_j))\subset\CC[3g-3]_\tau$ and $H^1(\compnu,T_{\compnu}(D_j))\subset\CC[3g-3]_u$. Therefore the proposition on page 28 and the arguments on page 36 of \cite{hamu1982} imply the following.
\begin{prop}\label{prop:comps}
Let $\varphi_j=\isoi^\nu_{|\compnu}$ be the induced automorphism on the normalisation $\compnu$ of the irreducible component $\comp$ of $\stab$. Then the pair $(\compnu,\varphi_j)$ is one of the following cases:
\begin{enumerate}
\item $\varphi_j=\identity_{\compnu}$, any $\compnu$
\item $\compnu$ is rational and $\ord\varphi_j=2,4$
\item $\compnu$ is elliptic and $\ord\varphi_j=2,4,3,6$
\item $\compnu$ is hyperelliptic of genus $2$ and $\varphi_j$ is the hyperelliptic involution
\item $\compnu$ is hyperelliptic of genus $3$ and $\varphi_j$ is the hyperelliptic involution
\item $\compnu$ is bielliptic of genus $2$, i.e. it is a double cover of an elliptic curve, and $\varphi_j$ is the associated involution
\end{enumerate}
\end{prop}
Now the possibility that $\isoi_\stab$ interchanges a pair of nodes can be excluded.
\begin{prop}
$\isoi_\stab$ fixes all nodes.
\end{prop}
\begin{proof}
Assume that $\isoi_\stab$ interchanges the nodes $P_1$ and $P_2$. Since $\isoi_\stab$ fixes all components there are only the following two possibilities:
{\scriptsize\[\begindc{0}[1]
\obj(-50,100)[]{(a)}[1]
\obj(0,90)[anu]{{\includegraphics{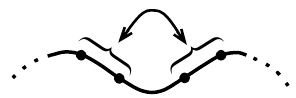}}}[1] 
\obj(50,80)[]{$\compnu$}[1] 
\obj(-22,81)[]{$P_1^+$}[1]
\obj(-10,75)[]{$P_1^-$}[1]
\obj(26,82)[]{$P_2^+$}[1]
\obj(12,75)[]{$P_2^-$}[1]

\obj(0,20)[a]{{\includegraphics{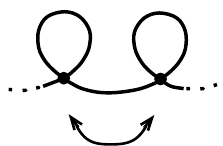}}}[1] 
\obj(-13,12)[]{$P_1$}[1] 
\obj(15,12)[]{$P_2$}[1]
\obj(40,20)[]{$\comp$}[1]  
\mor{anu}{a}[22,25]{$\nu_\stab$}[1,0] 

\obj(100,100)[]{(b)}[1]  
\obj(150,90)[bnu]{{\includegraphics{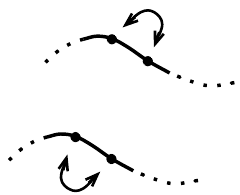}}}[1] 
\obj(192,95)[]{$\compnu$}[1] 
\obj(183,67)[]{$\compinu$}[1] 
\obj(146,100)[]{$P_1^+$}[1]
\obj(140,85)[]{$P_1^-$}[1]
\obj(157,93)[]{$P_2^+$}[1]
\obj(155,77)[]{$P_2^-$}[1]

\obj(150,20)[b]{{\includegraphics{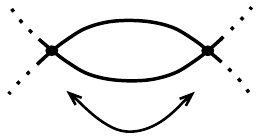}}}[1] 
\obj(129,18)[]{$P_1$}[1] 
\obj(173,18)[]{$P_2$}[1]
\obj(194,40)[]{$\comp$}[1]  
\obj(195,12)[]{$\compi$}[1]  
\mor{bnu}{b}[25,22]{$\nu_\stab$}[1,0] 
\enddc\]}

Here $(\compnu,\isoi^\nu_{|\compnu})$ and $(\compinu,\isoi^\nu_{|\compinu})$ are of the types (ii)-(vi) of Proposition~\ref{prop:comps}. In both cases $P_1$ and $P_2$ are non-disconnecting ($u_i=\tau_i$, $i=1,2$). The action of $\iso$ on $\CC_{u_1}\oplus\CC_{u_2}$ is $B=\left(\begin{smallmatrix}
0&\alpha_1\\ \alpha_2&0
\end{smallmatrix}\right)$ for appropriate non-zero scalars $\alpha_i$, $\alpha_1\alpha_2\neq 1$, since $(\spin,\iso)$ is singularity reduced and $\alpha_1\alpha_2=\zeta^{l\frac{2n}{\ord B}}$ for an appropriate $1\leq l<\frac{\ord B}{2}$. As in the proof of Proposition~\ref{prop:pairnodes} the eigenvalues are $\zeta^{l\frac{n}{\ord B}}$ and $\zeta^{\frac n2+l\frac{n}{\ord B}}$ and give a contribution to the \rsbt{} sum of $\frac 12+\frac{2l}{\ord B}$.

Since $\isoi_\stab$ is induced by the automorphism $\iso$ either both nodes are exceptional ($\tau_i^2=t_i$, $i=1,2$) or both are non-exceptional ($\tau_i=t_i$, $i=1,2$). In case $P_1, P_2\in\Delta$ the action of $\isoi_\stab$ on $\CC_{t_1}\oplus\CC_{t_1}$ is given by $B_\stab=B$. If $P_1, P_2\in\EX$ the action of $\isoi_\stab$ on $t_1$ and $t_2$ is given by $B_\stab=\left(\begin{smallmatrix}
0&\alpha_1^2\\ \alpha_2^2&0
\end{smallmatrix}\right)$. If $\ord B=\ord B_\stab$ the argument on page 37 of~\cite{hamu1982} gives a contradiction. Hence $P_1$ and $P_2$ are exceptional and $\ord B=2\ord B_\stab$. Note that the order of $B_\stab$ is even.

Consider case (a) and set $\varphi_j=\isoi^\nu_{|\compnu}$ and $n_j=\ord\varphi_j\in\{2,4,3,6\}$. Then $\ord B_\stab$ divides $n_j$, $n_j$ is even and 
\begin{align*}
&\ 1>\RT\geq \frac 12+\frac{2l}{\ord B}\geq\frac 12+\frac{2l}{2n_j}\\
\Longrightarrow &\ n_j>2l\geq 2\\
\Longrightarrow &\ n_j=4, 6
\end{align*} 
If $n_j=6$ Harris and Mumford prove that the action of $\isoi_\stab$ on $H^1(\compnu,T_{\compnu}(D_j))$ contributes $\frac 13$, hence
\[
1>\RT\geq \underbrace{\frac 12+\frac{2}{12}}_{P_1 \text{ and } P_2}+\underbrace{\frac 13}_{\compnu}=1.
\]
Therefore, $n_j=4$ and $\compnu$ is rational or elliptic and has at least the four marked points $P_1^\pm$, $P_2^\pm$. Then $\dim H^1(\compnu,T_{\compnu}(C_j))\geq 1$ and the action of $\isoi_\stab$ on this space contributes $\frac \kappa 4$ for an appropriate non-negative integer $\kappa$. Since $1>\RT\geq \frac 12+\frac{2}{8}+\frac \kappa 4$ this action has to be trivial. This in turn means that the automorphism of order $4$ interchanging $\{P_1^\pm\}\leftrightarrow\{P_2^\pm\}$ deforms to every deformation of the pointed curve $(\compnu,\{P_i^\pm\}\cap\compnu)$. But the general four-pointed rational or elliptic curve does not have an automorphism with these properties. Hence case (a) is impossible.

In case (b) set $\varphi_j=\isoi^\nu_{|\compnu}$, $\varphi_{j'}=\isoi^\nu_{|\compinu}$, $n_j=\ord\varphi_j\in\{2,4,3,6\}$, $n_{j'}=\ord\varphi_{j'}\in\{2,4,3,6\}$ and $\bar{n}=\lcm(n_j,n_{j'})\in\{2,3,4,6,12\}$. Then $\bar{n}=\ord \isoi^\nu_{\compnu\cup\compinu}$ and $\ord B_\stab$ divides $\bar n$.
\begin{align*}
&\ 1>\RT\geq \frac 12+\frac{2l}{\ord B}\geq\frac 12+\frac{1}{\ \bar n\ }\\
\Longrightarrow &\ \bar n=4,6,12
\end{align*} 
If $\bar n=6$ or $12$ Harris and Mumford calculate the contribution of the action on $H^1(\compnu,T_{\compnu}(D_j))$ and $H^1(\compinu,T_{\compinu}(D_{j'}))$, which leads to a contradiction in our case too. Hence $\bar n=4$ and wlog $n_j=4$ and $n_{j'}=2$ or $4$. $\varphi_j$ interchanges $P_1^+$ and $P_2^+$. An order $4$ automorphism on a rational curve does not have points of order $2$. Therefore, by Proposition~\ref{prop:comps} $\compnu$ is elliptic with $j$-invariant $1728$. The action of $\isoi_\stab$ on $H^1(\compnu,T_{\compnu}(D_j))$ then gives at least a contribution of $\frac 14$.
\[
1>\RT\geq \underbrace{\frac 12+\frac{2}{8}}_{P_1 \text{ and } P_2}+\underbrace{\frac 14}_{\compnu}=1.
\]
Hence case (b) is also impossible and all cases where $\isoi_\stab$ interchanges a pair of nodes are excluded.
\end{proof}
The next step is to refine Proposition~\ref{prop:comps}.
\begin{prop}\label{prop:compdet}
Let $\comp$ be an irreducible component of $\stab$ with normalization $\compnu$, $D_j$ the divisor of the marked points $\{P_i^\pm\}\cap\compnu$ and set $\varphi_j=\isoi^\nu_{|\compnu}$. Then $(\compnu,D_j,\varphi_j)$ is of one of the following types and the contribution to the \rsbt{} sum of the action of $\iso$ on $H^1(\compnu,T_{\compnu}(D_j))\subset\CC[3g-3]_u$ is at least $w_j$.

\begin{description}
\item[\textnormal{\textit{Identity}}] $\varphi_j=\identity_{\compnu}$, any $(\compnu,D_j)$, $w_j=0$
\item[\textnormal{\textit{Elliptic tail}}] $\compnu$ is elliptic, $D_j$ is of the form $D_j=P_1^+$ and $P_1^+$ is fixed by~$\varphi_j$.
\begin{description}
\item[\textnormal{\textit{order 2}}] $\ord\varphi_j=2$, $w_j=0$
\item[\textnormal{\textit{order 4}}] $\compnu$ has $j$-invariant $1728$, $\ord\varphi_j=4$, $w_j=\frac 12$
\item[\textnormal{\textit{order 3}}] $\compnu$ has $j$-invariant $0$, $\ord\varphi_j=3$, $w_j=\frac 13$
\item[\textnormal{\textit{order 6}}] $\compnu$ has $j$-invariant $0$, $\ord\varphi_j=6$, $w_j=\frac 13$
\end{description}
\item[\textnormal{\textit{Elliptic ladder}}] $\compnu$ is elliptic, $D_j$ is of the form $D_j=P_1^++P_2^+$, $P_1^+$ and $P_2^+$ are fixed by $\varphi_j$.
\begin{description}
\item[\textnormal{\textit{order 2}}] $\ord\varphi_j=2$, $w_j=\frac 12$
\item[\textnormal{\textit{order 4}}] $\compnu$ has $j$-invariant $1728$, $\ord\varphi_j=4$, $w_j=\frac 34$
\item[\textnormal{\textit{order 3}}] $\compnu$ has $j$-invariant $0$, $\ord\varphi_j=3$, $w_j=\frac 23$
\end{description}
\item[\textnormal{\textit{Hyperelliptic tail}}] $\compnu$ has genus $2$, $\varphi_j$ is the hyperelliptic involution, $D_j$ is of the form $D_j=P_1^+$, $P_1^+$ is fixed by $\varphi_j$ and $w_j=\frac 12$.
\end{description}
\end{prop}
\begin{proof}
Since all components of $\stab$ are fixed by $\isoi_\stab$, the action of $\isoi_\stab$ on $\CC[3g-3]_t$ restricts to an action on $H^1(\compnu,T_{\compnu}(D_j))$ and this action coincides with that of $\iso$ on $H^1(\compnu,T_{\compnu}(D_j))\subset\CC[3g-3]_u$. Moreover, since all nodes of $\stab$ are fixed by $\isoi_\stab$ a point $P_i^\pm\in\compnu$ is either a fixed point of $\varphi_j$ or $P_i^+$ and $P_i^-$ both lie in $\compnu$ and are interchanged by $\varphi_j$. 

Consider the six cases of Proposition~\ref{prop:comps}. In case~(i) $\varphi_j$ is the identity, $(\compnu,D_j)$ is arbitrary and the action on $H^1(\compnu,T_{\compnu}(D_j))$ is trivial, hence $w_j=0$. In case~(ii) $\compnu$ is rational, $\ord \varphi_j=2$ or $4$ and there are at least three marked points on $\compnu$, since $\stab$ is stable. If $\ord\varphi_j=4$ there are exactly two fixed points on $\compnu$ while all other points are of order $4$, giving the contradiction $\deg D_j\leq 2$. If $\ord\varphi_j=2$ there are exactly two fixed points, all other points are of order two. Since $\deg D_j\geq 3$ there is at least one pair of marked points $P_1^+$ and $P_1^-$ on $\compnu$ which are interchanged by $\varphi_j$. Let $xy=0$ be a local equation for $\stab$ at $P_1$, then $\isoi_\stab$ acts as $t_1=xy\mapsto yx=t_1$. Therefore, $\iso$ acts as $\tau_1\mapsto\pm\tau_1$ and since $P_1$ is non-disconnecting and $(\spin,\iso)$ singularity reduced, we must have $\tau_1\mapsto-\tau_1$ and $P_1$ is exceptional. Moreover, $u_1=\tau_1\mapsto-\tau_1=-u_1$ giving a contribution of $\frac 12$ to the \rsbt{} sum. Since $1>\RT \geq (\text{number of pairs interchanged})\cdot\frac 12$, the pair $P_1^\pm$ is the only pair of marked points on $\compnu$ which are interchanged and $\deg D_j\leq 4$. If $\deg D_j=4$ the action on $H^1(\compnu,T_{\compnu}(D_j))$ gives a contribution of at least $\frac 12$, since the order $2$ automorphism $\varphi_j$ does not deform to the general four-pointed rational curve. Hence we get the contradiction $1>\RT\geq \frac 12+\frac 12$ and $\deg D_j$ has to be $3$, implying that $\comp$ is a singular elliptic tail. Consider the restriction of the local universal deformation $(\calX\rightarrow\CC[3g-3]_\tau,\calL,\calB)$ to the locus $\{\tau_k=0|k\neq 1\}$, i.e. the singular elliptic tail $\comp$ is smoothed to a smooth elliptic tail $\comp'$. But then the line bundle over the singular elliptic tail deforms to a theta characteristic on $\comp'$, the automorphism $\varphi_j$ deforms to the elliptic involution and hence $\iso$ acts trivially on $\tau_1$, in contradiction to $\tau_1\mapsto-\tau_1$. Therefore, the case~(ii) is excluded. 
{\scriptsize
\[\begindc{0}[1]
\obj(0,35)[]{central fibre}[1] 
\obj(120,35)[]{nearby fibre}[1] 
\obj(0,0)[central]{{\includegraphics{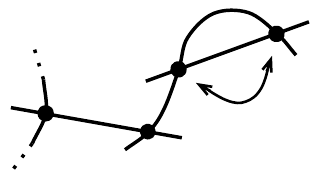}}}[1] 
\obj(120,0)[near]{{\includegraphics{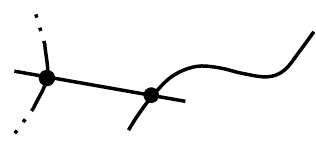}}}[1] 
\obj(-1,11){$P_1^+$}
\obj(47,6){$\comp$}
\obj(39,25){$P_1^-$}
\enddc\]}

In case~(iii) $\compnu$ is elliptic, $\ord\varphi_j=2,4,3$ or $6$ and $\deg D_j\geq 1$. Let us first consider the case that $\varphi_j$ is a translation. Then $\varphi_j$ has no fixed points, $D_j=P_1^++P_1^-+\dotsb+P_k^++P_k^-$, $P_i^+\leftrightarrow P_i^-$ and hence $\varphi_j$ has to be of order $2$. Therefore, for every $i=1,\dotsc,k$ the action of $\isoi_\stab$ is $t_i=xy\mapsto yx=t_i$, where $xy=0$ is a local equation of $\stab$ at $P_i$, and $u_i\mapsto-u_i$ as above, giving a contribution of $\frac 12$ to the \rsbt{} sum. Since $\stab$ is connected and has at least genus $4$ there are at least three such pairs of marked points on $\compnu$, giving a contribution of at least $\frac 32$ and this case is excluded. Hence $\varphi_j$ has a fixed point. 

In case $\ord\varphi_j=2$ consider the case where there is a pair of marked points such that $P_1^+\leftrightarrow P_1^-$. Again this pair contributes $\frac 12$ and, therefore, it is the only one. But the elliptic involution does not deform to the general $l$-pointed curve if $l\geq 2$, hence the action on $H^1(\compnu,T_{\compnu}(D_j))$ is non-trivial and contributes at least $\frac 12$, giving a \rsbt{} sum of at least $1$. Thus, all marked points are fixed points of the elliptic involution $\varphi_j$ and $1\leq\dim H^1(\compnu,T_{\compnu}(D_j))=\deg D_j\leq 4$. The fixed point locus of the action of $\varphi_j$ on $H^1(\compnu,T_{\compnu}(D_j))$ is one-dimensional, since we may deform the elliptic curve with the marked point $P_1^+$, but then the remaining marked points are fixed, since we want the elliptic involution to deform. Hence, the action of $\varphi_j$ on $H^1(\compnu,T_{\compnu}(D_j))$ has $-1$ as an eigenvalue of multiplicity $\deg D_j-1$ and contributes $(\deg D_j -1)\cdot\frac 12$ to the \rsbt{} sum. So either $\deg D_j=1$ and we are in the elliptic tail case of order $2$ with $w_j=0$ or $\deg D_j=2$ and we are in the elliptic ladder case of order $2$ with $w_j=\frac 12$.

In case $\ord\varphi_j=4$, the $j$-invariant of $\compnu$ is $1728$ and there are two fixed points and one pair of order two points on $\compnu$. Suppose there is a pair of marked points such that $P_1^+\leftrightarrow P_1^-$. Then the action of $\iso$ on $u_1$ contributes $\frac 12$ as above. Moreover, there exists a one-dimensional deformation of $(\compnu,D_j)$ to which $\varphi_j^2$, the elliptic involution, deforms, hence $-1$ is an eigenvalue of the action of $\varphi_j$ on $H^1(\compnu,T_{\compnu}(D_j))$, also contributing $\frac 12$. Hence, all marked points on $\compnu$ have to be fixed points of $\varphi_j$ and $1\leq\deg D_j\leq 2$. If there is only one marked point, we are in the elliptic tail case of order $4$ and the action on $H^1(\compnu,T_{\compnu}(D_j))$ contributes $w_j=\frac 12$. If we have two marked points, we are in  the elliptic ladder case of order $4$. Then $\dim H^1(\compnu,T_{\compnu}(D_j))=2$ and the action of $\varphi_j^2$ has eigenvalues $1$ and $-1$. Hence $\varphi_j$ has eigenvalues $-1$ and $\xi_4$, a primitive fourth root of $1$, giving a contribution of $w_j=\frac 34$.

In case $\ord\varphi_j=3$, $\compnu$ has $j$-invariant $0$ and $\varphi_j$ has three fixed points and no points of order $2$. The action on $H^1(\compnu,T_{\compnu}(D_j))$ does not have the eigenvalue $1$, since the order three automorphism does not deform to any deformation of the pointed elliptic curve. Therefore, all eigenvalues are primitive third roots of $1$ and the contribution is at least $\frac 13\dim H^1(\compnu,T_{\compnu}(D_j))=\frac 13\deg D_j$. Hence three marked points are not possible, one marked point gives the elliptic tail case ($w_j=\frac 13$) and two marked points give the elliptic ladder case ($w_j=\frac 23$).

In case $\ord\varphi_j=6$, $\compnu$ has $j$-invariant $0$ and $\varphi_j$ has one fixed point and one pair of order two points. If there is a pair of marked points such that $P_1^+\leftrightarrow P_1^-$ the action on $u_1$ contributes $\frac 12$. Moreover, the action of $\varphi_j^3$, the elliptic involution, on $H^1(\compnu,T_{\compnu}(D_j))$ has $1$ and $-1$ as eigenvalues, while $\varphi_j$ does not deform to any deformation of the pointed elliptic curve. Hence, $\varphi_j$ has a primitive third root and a non-trivial sixth root of $1$ as eigenvalues, alltogether giving a contribution of $\frac 12+\frac 13+\frac 16=1$. Therefore, we are left with one marked point, which is fixed by $\varphi_j$, giving the elliptic tail case. The action on $H^1(\compnu,T_{\compnu}(D_j))$ contributes at least $\frac 13$, since the action of $\varphi_j^3$ on this space is trivial.

Consider now the remaining cases~(iv)-(vi). Either $\compnu$ is hyperelliptic of genus $2$ or $3$ , hence there is a $2:1$-map $\compnu\rightarrow\PP[1]=B$, which is ramified in $r=6$ resp. $r=8$ points, or $\compnu$ is bielliptic of genus $2$, hence there is a $2:1$-map $\compnu\rightarrow B$, which is ramified in $r=2$ points and the base curve $B$ is an elliptic curve. In all cases $\varphi_j$ is the associated involution to the map $\compnu\rightarrow B$ and there are $r$ fixed points, while all other points are of order $2$. Assume that there exists a pair of marked points such that $P_1^+\leftrightarrow P_1^-$. Then the action on $u_1$ contributes $\frac 12$. Moreover, the action on $H^1(\compnu,T_{\compnu}(D_j))$ is non-trivial, since $\varphi_j$ cannot deform to every deformation of the pointed curve $(\compnu,D_j)$, hence this action also contributes at least $\frac 12$ and the case is excluded. Thus, all marked points are fixed points of $\varphi_j$ and $1\leq\deg D_j\leq r$. We have
\[
\dim H^1(\compnu,T_{\compnu}(D_j))=3g(\compnu)-3+\deg D_j=\begin{cases}
3+\deg D_j & \text{ if } g(\compnu)=2\\
6+\deg D_j & \text{ if } g(\compnu)=3
\end{cases}
\]
The action of $\varphi_j$ on this space has eigenvalues $1$ and $-1$. The dimension of the eigenspace with respect to the eigenvalue $1$ is the dimension of the deformation space of the $r$-pointed curve $B$ with marked points the branch points, hence
\[
\dim \eig(\varphi_j,1)=3g(B)-3+r=\begin{cases}
3 & \text{ if } \compnu \text{ is hyperelliptic of genus }2\\
2 & \text{ if } \compnu \text{ is bielliptic}\\
5 & \text{ if } \compnu \text{ is hyperelliptic of genus }3
\end{cases}
\]
and
\begin{align*}
\dim \eig(\varphi_j,-1)&=\dim H^1(\compnu,T_{\compnu}(D_j))-\dim \eig(\varphi_j,1)\\
&=\begin{cases}
\deg D_j & \text{ if } \compnu \text{ is hyperelliptic of genus }2\\
1+\deg D_j & \text{ if } \compnu \text{ is bielliptic}\\
1+\deg D_j & \text{ if } \compnu \text{ is hyperelliptic of genus }3
\end{cases}
\end{align*}
In the latter two cases the contribution to the \rsbt{} sum is $\frac 12(1+\deg D_j)\geq 1$, hence these cases are excluded. In case $g(\compnu)=2$ and $\varphi_j$ is the hyperelliptic involution, we get $1>\RT\geq \frac 12\deg D_j$ and hence we are in the case of a hyperelliptic tail and $w_j=\frac 12$.
\end{proof}
In the next propositions we will patch together several components of the types in the above proposition.
\begin{prop}\label{prop:hyp}
The hyperelliptic tail case is impossible.
\end{prop}
\begin{proof}
Let $\comp$ be a hyperelliptic tail of genus $2$ of $\stab$, $\varphi_j=\isoi^\nu_{|\compnu}$ the hyperelliptic involution, $P_1$ the node where $\comp$ meets the rest of the curve and $\compi$ the second component through $P_1$. Then $\comp$ contributes $\frac 12$ to the \rsbt{} sum. $\compi$ is of one of the cases of Proposition~\ref{prop:compdet}. If $\compi$ is a hyperelliptic tail it also contributes $\frac 12$ and this gives the contradiction $1>\RT\geq \frac 12+\frac 12$. If $\compi$ is an elliptic ladder, its contribution is $\frac 12$, $\frac 23$ or $\frac 34$. In any case the \rsbt{} sum is at least $1$, excluding these cases. If $\compi$ is an elliptic tail, the curve $\stab$ has only genus $3$. Therefore, we are left with the case that $\compi$ is an identity component. Let $xy=0$ be a local equation for $\stab$ at $P_1$, then $\isoi_\stab$ acts as $x\mapsto-x$ and $y\mapsto y$, hence $t_1\mapsto-t_1$. The node $P_1$ is exceptional ($\tau_1^2=t_1$) and it is disconnecting but not an elliptic tail node ($u_1=\tau_1^2$). Therefore, $\iso$ acts as $u_1\mapsto-u_1$ giving a contribution of $\frac 12$ and $1>\RT\geq \frac 12+\frac 12$. In conclusion, the case of a hyperelliptic tail is impossible.
\end{proof}
\begin{prop}
The elliptic ladder cases are impossible.
\end{prop}
\begin{proof}
Let $\comp$ be an elliptic ladder, i.e. $\compnu$ is elliptic with two marked points $P_1^+$ and $P_2^+$ and $\varphi_j=\isoi^\nu_{|\compnu}$ is an automorphism of order $2$, $3$ or $4$ fixing the marked points. Denote by $\compi$ resp. $\compii$ the second component through $P_1$ resp. $P_2$, these have to be identity components, elliptic tails or elliptic ladders by Propositions~\ref{prop:compdet} and \ref{prop:hyp}. Every elliptic ladder contributes $\frac 12$, $\frac 23$ or $\frac 34$ depending on whether $\ord \varphi_j=2$, $3$ or $4$. Therefore, two elliptic ladders give a contribution of at least $1$ and neither $\compi$ nor $\compii$ is an elliptic ladder. If $\compi$ and $\compii$ were both elliptic tails, the curve would have genus $3$. Hence, wlog we may assume that $\compi$ is an identity component.

Let $xy=0$ be a local equation for $\stab$ at $P_1$, then $\isoi_\stab$ acts as $x\mapsto\alpha x$ and $y\mapsto y$, where $\alpha$ is a primitive $\ord\varphi_j$th root of $1$, and $t_1\mapsto\alpha t_1$. $P_1$ is either disconnecting and exceptional ($u_1=\tau_1^2=t_1$), non-disconnecting and exceptional ($u_1=\tau_1$, $\tau_1^2=t_1$) or non-disconnecting and non-exceptional ($u_1=\tau_1=t_1$). Therefore, the action of $\iso$ on $u_1$ is 
\[
u_1\mapsto\begin{cases}
-u_1 & \text{ if } u_1=t_1 \text{ and } \ord\varphi_j=2\\
\zeta_3u_1 & \text{ if } u_1=t_1 \text{ and } \ord\varphi_j=3\\
\zeta_4u_1 & \text{ if } u_1=t_1 \text{ and } \ord\varphi_j=4\\
\zeta_4u_1 & \text{ if } u_1^2=t_1 \text{ and } \ord\varphi_j=2\\
\zeta_3u_1 \text{ or }\zeta_6u_1 & \text{ if } u_1^2=t_1 \text{ and } \ord\varphi_j=3\\
\zeta_8u_1 & \text{ if } u_1^2=t_1 \text{ and } \ord\varphi_j=4
\end{cases}
\] 
where $\zeta_k$ denotes a primitive $k$th root of $1$. The case $u_1=t_1$ is impossible, since the contribution of the action on $H^1(\compnu,T_{\compnu}(D_j))$ and $u_1$ is at least
\[
\begin{cases}
\frac 12 + \frac 12 & \text{ if }\ord\varphi_j=2\\
\frac 23 + \frac 13 & \text{ if }\ord\varphi_j=3\\
\frac 34 + \frac 14 & \text{ if }\ord\varphi_j=4
\end{cases}
\]
Hence, $P_1$ is exceptional and non-disconnecting. But if $P_1$ is non-diconnecting, the second node $P_2$ on $\comp$ is also non-disconnecting and $\compii$ cannot be an elliptic tail but has to be an identity component. Out of degree reasons $P_2$ has to be an exceptional node. Therefore, the action on $u_2$ is analogous to that on $u_1$ and the actions on $H^1(\compnu,T_{\compnu}(D_j))$, $u_1$ and $u_2$ contribute
\[
\begin{cases}
\frac 12 + \frac 14 +\frac 14 & \text{ if }\ord\varphi_j=2\\
\frac 23 + \frac 16 + \frac 16 & \text{ if }\ord\varphi_j=3\\
\frac 34 + \frac 18 + \frac 18 & \text{ if }\ord\varphi_j=4
\end{cases}
\]
In conclusion the elliptic tail case is not possible. 
\end{proof}
\begin{prop}
The elliptic tail case of order $4$ is impossible
\end{prop}
\begin{proof}
Let $\comp$ be an elliptic tail of order $4$, i.e., $\compnu$ is elliptic with one marked point $P_1^+$ and $\varphi_j=\isoi^\nu_{|\compnu}$ is of order $4$ and fixes $P_1^+$. The second component through $P_1$ cannot be an elliptic tail, since if so we would have $g=2$, hence it is an identity component and $t_1\mapsto\zeta_4 t_1$ as in the proof of the last proposition. $P_1$ is an elliptic tail node and exceptional, hence $u_1=\tau_1^4=t_1^2$ and $u_1\mapsto-u_1$, giving a contribution of $\frac 12$. Together with the contribution of $\frac 12$ of the action on $H^1(\compnu,T_{\compnu}(D_j))$, we get a \rsbt{} sum of at least $1$.
\end{proof}
\begin{prop}
The case where $\stab$ has \emph{no} elliptic tail of order $3$ or $6$ is impossible.
\end{prop}
\begin{proof}
Suppose the action of $\isoi_\stab$ on $\stab$ only has identity components and elliptic tails of order $2$. Then the action of $\isoi_\stab$ on $H^1(\compnu,T_{\compnu}(D_j))$ is trivial for every irreducible component of $\stab$ (see Proposition~\ref{prop:compdet}). Moreover, the action of $\iso$ on a coordinate $u_i$ corresponding to an elliptic tail node is trivial, since $t_i=xy\mapsto -xy=-t_i$, $u_i=\tau_i^4=t_1^2$ and hence $u_i\mapsto u_i$. If $t_i$ corresponds to any other node $t_i\mapsto t_i$ since the components meeting at such a node are identity components. In any case $u_i\mapsto\pm u_i$ and the action of $\iso$ on $\CC[3g-3]_u$ is either the identity or is of order $2$. Then the \rsbt{} sum does not fulfil the inequality $0<\RT<1$, since if an order $2$ action fulfils this inequality it has to be a quasireflection, but no element of $\aut\spin$ acts on $\CC[3g-3]_u$ as a quasireflection. Thus, $\stab$ has to have an elliptic tail of order $3$ or $6$.
\end{proof}
\begin{proof}[Proof of the only-if-part of Theorem~\ref{thm:cansing}]
We proved, that if $(\spin,\iso)$ is singularity reduced and the action of $\iso$ on $\CC[3g-3]_u$ has \rsbt{} sum $0<\RT<1$ with respect to the primitive root $\zeta$, then the stable model $\stab$ has an elliptic tail $\comp$ of order $3$ or $6$, i.e., $\compnu$ is an elliptic curve with $j$-invariant $0$, $\isoi_\stab$ fixes the elliptic tail node and has order $3$ resp. $6$ on $\comp$. Such an automorphism lifts to the spin curve if and only if the theta characteristic $\spinii_j^\nu=\nu^*\spinii_{|\compnu}$ is trivial. Moreover, every elliptic tail together with its elliptic tail node contributes at least $\frac 23$ to the \rsbt{} sum. Hence, $\stab$ has exactly one elliptic tail of order $3$ or $6$ and $\isoi_\stab$ is the identity on the remaining components.

If $[\spin]\in\sgbar$ is a non-canonical singularity, then there exists an automorphism $\iso$ whose action on $\CC[3g-3]_u$ fulfils $0<\RT<1$ with respect to a primitive root $\zeta$. By Proposition~\ref{prop:singred} we can deform the pair $(\spin,\iso)$ to a singularity reduced pair $((\spini',\spinii',\spiniii'),(\isoi',\isoii'))$. The eigenvalues of the actions of $\iso$ and $(\isoi',\isoii')$ on the respective $u$-coordinates coincide, hence $(\isoi',\isoii')$ also fulfils $0<\frac 1n\sum_{j=0}^{3g-3}a_j'<1$ with respect to $\zeta$. Thus, $\stab'$ has an elliptic tail of order $3$ or $6$ and ${\spinii'_j}^\nu$ is trivial. In the proof of Proposition~\ref{prop:singred} we showed, that this property is not affected by the deformation. Therefore, $\stab$ has an elliptic tail with $j$-invariant $0$ and $\spinii_j^\nu$ is trivial and this finishes the proof.
\end{proof}
\begin{cor}\label{cor:bad}
Let $\iso$ be an automorphism of $\spin$ whose action on $\CC[3g-3]_u$ has \rsbt{} sum $0<\RT<1$ for an appropriate primitive root $\zeta$. Then $\stab$ has exactly one elliptic tail on which $\isoi_\stab$ acts of order $3$ or $6$ and $\isoi_\stab$ is the identity on all other components.
\end{cor}

\section{Pluricanonical forms}
We will prove the following lifting result for pluricanonical forms, which is the analogue of Theorem~1 of~\cite{hamu1982}.
\begin{thm}\label{thm:lift}
Let $\widetilde S_g\rightarrow\sgbar$ be a desingularisation of $\sgbar$ and $g\geq 4$. Every pluricanonical form $\omega$ on the regular locus $\sgbar^{\text{reg}}$ extends holomorphically to $\widetilde S_g$, i.e., for any $k$
\[
\Gamma\left(\sgbar^{\text{reg}},kK_{\sgbar}\right)\cong\Gamma\left(\widetilde S_g, kK_{\widetilde S_g}\right).
\] 
\end{thm}
We will need the following generalisation of the \rsbt{} criterion by Harris and Mumford.
\begin{prop}\emph{[see pp.~27~f. and Appendix~1 to~\S 1 of~\cite{hamu1982}]}
Let $V=\CC[m]$, $G\subset\GL(V)$ a finite group, $\widetilde{V/G}\rightarrow V/G$ a desingularisation and $\omega$ a $G$-invariant pluricanonical form on $V$. Denote by $V^0\subset V$ the set where $G$ acts freely and by $\fix(M)\subset V$ the fixed point set of $M\in G$. Let $U\subset V/G$ be an open set containing $V^0/G$ such that for every $M\in G$ whose \rsbt{} sum fulfils $0<\RT<1$ for some primitive $\ord M$th root $\zeta$ the intersection of $U$ with the image of $\fix(M)$ in $V/G$ is non-empty. Denote by $\widetilde U\subset\widetilde{V/G}$ its preimage under the desingularisation. If $\omega$ considered as a meromorphic form on $\widetilde{V/G}$ is holomorphic on $\widetilde U$ then it is holomorphic on $\widetilde{V/G}$.
\end{prop}
In order to prove Theorem~\ref{thm:lift} we have to further analyse the geometry of $\sgbar$ near the non-canonical singularities. In the case of $\mgbar$ Harris and Mumford constructed an open neighbourhood $S$ of the moduli point $[\stab]$ of a general curve $\stab$ with a given elliptic tail $\stab_1$. The curve $\stab$ has two irreducible components, the elliptic tail $\stab_1$ and a smooth curve $\stab_2$ of genus $g-1$ without automorphisms meeting at one node $P$. Denote by $P^+\in\stab_1^\nu$ and $P^-\in\stab_2^\nu$ the preimages of the node $P$ under the normalisation. Consider the map
\begin{align*}
\psi:\PP[1]=\mgbar[1,1]&\longrightarrow\mgbar\\
[(C_1',{P'}^+)]&\longmapsto[\stab']
\end{align*}
where $\stab'$ is the stable curve with irreducible components $\stab_1'$ and $\stab_2$ glued at ${P'}^+$ and $P^-$. Then $\psi$ is an isomorphism onto its image and the open neighbourhood $S=S(\stab_2,P^-)$ of $[\stab]$ has the following properties (see pp.~40~ff. in~\cite{hamu1982}):
\begin{enumerate}
\item $S$ contains the image $\im\psi$.
\item There exists a birational morphism $S\rightarrow B$ where $B$ is smooth and $3g-3$-dimensional.
\item There exists a subvariety $Z\subset B$ of codimension two such that the preimage of $B\setminus Z$ under the map $S\rightarrow B$ is isomorphic to $B\setminus Z$.
\item A point in $B\setminus Z\subset S\subset\mgbar$ corresponds to an irreducible stable curve with trivial automorphism group, i.e., $B\setminus Z\subset\mgbar^0$, where $\mgbar^0$ is the locus of curves with trivial automorphism group.
\[\begindc{0}[1]
\obj(0,0)[lu]{$B$}[1] 
\obj(70,0)[mu]{$B\setminus Z$}[1] 
\obj(140,0)[ru]{$\mgbar^0$}[1] 
\obj(70,40)[mm]{$S$}[1] 
\obj(140,40)[rm]{$\mgbar$}[1] 
\obj(70,80)[mo]{$\im\psi$}[1] 
\obj(140,80)[ro]{$\Delta_1$}[1] 
\mor{mu}{lu}[10,10]{$$}[1,3] 
\mor{mu}{ru}[10,10]{$$}[1,3] 
\mor{mm}{lu}[10,10]{$$}[1,0] 
\mor{mm}{rm}[10,10]{$$}[1,3] 
\mor{mu}{mm}[10,10]{$$}[1,3] 
\mor{ru}{rm}[10,10]{$$}[1,3] 
\mor{mo}{ro}[10,10]{$$}[1,3] 
\mor{mo}{mm}[10,10]{$$}[1,3] 
\mor{ro}{rm}[10,10]{$$}[1,3] 
\enddc\]
\end{enumerate}
The importance of this construction is the following. If $\omega$ is a pluricanonical form on the regular locus of $S$ and $\widetilde S\rightarrow S$ a desingularisation, then $\omega$ restricts to a form on $B\setminus Z$, extends over the codimension two locus $Z$ to $B$ since $B$ is smooth and the pullback of this extension via $\widetilde S\rightarrow S\rightarrow B$ is holomorphic. 
\begin{proof}[Proof of Theorem~\ref{thm:lift}]
Let $\omega$ be a pluricanonical form on $\sgbar^{\text{reg}}$. It is enough to prove that $\omega$ lifts to a desingularisation of an open neighbourhood of every point $[\spin]\in\sgbar$.

\emph{First case: $[\spin]$ is a canonical singularity.} By definition every pluricanonical form on a small neighbourhood of a canonical singularity lifts to a desingularisation.

\emph{Second case: $[\spin]$ is a general non-canonical singularity.} It follows from Theorem~\ref{thm:cansing} that the stable model $\stab$ of $\spini$ has two irreducible components $\stab_1$ and $\stab_2$ meeting in one node $P$, where $\stab_1$ is a smooth elliptic curve with $j$-invariant $0$ and $\stab_2$ is a smooth curve of genus $g-1$ with trivial automorphism group. Denote the preimages of $P$ under the normalisation by $P^+\in\stab_1^\nu$ and $P^-\in\stab_2^\nu$. Moreover, $\spinii_i^\nu=\nu_\spini^*\spinii_{|\stab_i^\nu}$, $i=1,2$, is a theta characteristic and $\spinii_1^\nu$ is trivial. The map $\psi:\mgbar[1,1]\rightarrow\mgbar$ can be lifted to $\sgbar$ as follows:
\begin{align*}
\widehat\psi:\mgbar[1,1]&\longrightarrow\sgbar\\
[(\stab_1',{P'}^+)]&\longmapsto[(\spini',\spinii',\spiniii')]
\end{align*}
where the isomorphism class $[(\spini',\spinii',\spiniii')]$ is uniquely determined by requiring $\spini'\rightarrow\stab'$ to be the blow up at the node $P'$ of the stable curve $\stab'$ with irreducible components $\stab_1'$ and $\stab_2$ glued by ${P'}^+=P^-$, ${\spinii'_1}^\nu=\calO_{{\stab'_1}^\nu}$ and ${\spinii'_2}^\nu=\spinii_2^\nu$. 

Since $\aut\stab_2$ is trivial and ${\spinii'_1}^\nu=\calO_{{\stab'_1}^\nu}$ every automorphism of $\stab'$ lifts to $(\spini',\spinii',\spiniii')$ and 
\[
0\longrightarrow\auto(\spini',\spinii',\spiniii')\longrightarrow\aut(\spini',\spinii',\spiniii')\longrightarrow\aut\stab'\longrightarrow 0
\]
is exact. The action of $\aut\stab'$ on the deformation space $\CC[3g-3]_t$ is
\[
\begin{cases}
\left\langle\left(\begin{smallmatrix}
-1&\\
&\1
\end{smallmatrix}\right)\right\rangle&\text{ if }\aut\stab'\cong\ZZ_2\\[1ex]
\left\langle\left(\begin{smallmatrix}
\zeta_4&&\\
&-1&\\
&&\1
\end{smallmatrix}\right)\right\rangle&\text{ if }\aut\stab'\cong\ZZ_4\\[1ex]
\left\langle\left(\begin{smallmatrix}
\zeta_6&&\\
&\zeta_6^2&\\
&&\1
\end{smallmatrix}\right)\right\rangle&\text{ if }\aut\stab'\cong\ZZ_6
\end{cases}
\]
Again we may devide out by the subgroup generated by the quasireflection $\left(\begin{smallmatrix}
-1&\\
&\1
\end{smallmatrix}\right)$ and get $\CC[3g-3]_t/\aut\stab'\cong\CC[3g-3]_{u'}/\paut\,\stab'$, where $u_1'=t_1^2$, $u_i'=t_i$, $i=2,\dotsc,3g-3$ and
\[
\paut\,\stab'=\begin{cases}
\{\1\}&\text{ if }\aut\stab'\cong\ZZ_2\\[1ex]
\left\langle\left(\begin{smallmatrix}
-1&&\\
&-1&\\
&&\1
\end{smallmatrix}\right)\right\rangle&\text{ if }\aut\stab'\cong\ZZ_4\\[1ex]
\left\langle\left(\begin{smallmatrix}
\zeta_3&&\\
&\zeta_3&\\
&&\1
\end{smallmatrix}\right)\right\rangle&\text{ if }\aut\stab'\cong\ZZ_6
\end{cases}
\]
Moreover, $P'$ is the only node of $\stab'$ and it is an elliptic tail node, hence $u_1=\tau_1^4=t_1^2$, $u_i=\tau_i=t_i$, $i=2,\dotsc,3g-3$ and
\[
\paut(\spini',\spinii',\spiniii')=\begin{cases}
\{\1\}&\text{ if }\aut\stab'\cong\ZZ_2\\[1ex]
\left\langle\left(\begin{smallmatrix}
-1&&\\
&-1&\\
&&\1
\end{smallmatrix}\right)\right\rangle&\text{ if }\aut\stab'\cong\ZZ_4\\[1ex]
\left\langle\left(\begin{smallmatrix}
\zeta_3&&\\
&\zeta_3&\\
&&\1
\end{smallmatrix}\right)\right\rangle&\text{ if }\aut\stab'\cong\ZZ_6
\end{cases}
\]
Therefore, locally at the point $[(\spini',\spinii',\spiniii')]\in\im\widehat\psi$ the forgetful morphism $\pi:\sgbar\rightarrow\mgbar$ is the isomorphism $\CC[3g-3]_u/\paut(\spini',\spinii',\spiniii')\cong\CC[3g-3]_{u'}/\paut\,\stab'$. This implies that the open neighbourhood $S=S(\stab_2,P^-)$ of $\im\psi\subset\mgbar$ lifts to an open neighbourhood $S(\stab_2,\spinii_2^\nu,P^-)\xrightarrow[\pi]{\ \cong\ } S$ of $\im\widehat\psi$ (after shrinking $S$ if necessary). A point in $B\setminus Z\subset S\subset\sgbar$ corresponds to a spin curve whose stable model is irreducible and has trivial automorphism group, hence the spin curve has automorphism group $\{(\identity,\pm\identity)\}$ and $B\setminus Z\subset \sgbar^{\text{reg}}$ giving the following diagram
\[\begindc{0}[1]
\obj(0,0)[lu]{$B$}[1] 
\obj(70,0)[mu]{$B\setminus Z$}[1] 
\obj(140,0)[ru]{$\sgbar^{\text{reg}}$}[1] 
\obj(70,40)[mm]{$S$}[1] 
\obj(140,40)[rm]{$\sgbar$}[1] 
\mor{mu}{lu}[10,10]{$$}[1,3] 
\mor{mu}{ru}[10,10]{$$}[1,3] 
\mor{mm}{lu}[10,10]{$$}[1,0] 
\mor{mm}{rm}[10,10]{$$}[1,3] 
\mor{mu}{mm}[10,10]{$$}[1,3] 
\mor{ru}{rm}[10,10]{$$}[1,3] 
\enddc\]
The form $\omega$ on $\sgbar^{\text{reg}}$ restricts to $B\setminus Z$, extends holomorphically over the codimension two locus $Z$ to all of $B$ since $B$ is smooth, and this extension lifts holomorphically to a desingularisation $\widetilde S\rightarrow S$ via the concatenation $\widetilde S\rightarrow S\rightarrow B$.

\emph{Third case: $[\spin]\in\sgbar$ is any non-canonical singularity.} Denote by $\stab_1^{(1)},\dotsc,\stab_1^{(k)}$ all elliptic tails with $j$-invariant $0$ and $\spinii_1^{(i),\nu}=\nu^*\spinii_{|\stab_1^{(i),\nu}}$ trivial. Let $\stab_2^{(i)}=\overline{\stab\setminus\stab_1^{(i)}}$, $P^{(i),+}\in\stab_1^{(i),\nu}$ and $P^{(i),-}\in\stab_2^{(i),\nu}$ the preimages of the node $P^{(i)}$ connecting $\stab_1^{(i)}$ and $\stab_2^{(i)}$ under the normalisation. For every $i$ there is a deformation of $\spin$ to a spin curve $({\spini'}^{(i)},{\spinii'}^{(i)},{\spiniii'}^{(i)})$ that preserves the elliptic tail $\stab_1^{(i)}$ and the trivial theta characteristic on it and is general with this property. In particular, the stable model ${\stab'}^{(i)}$ of the deformed spin curve has two irreducible components, the elliptic tail $\stab_1^{(i)}$ and a smooth curve ${\stab'_2}^{(i)}$ with trivial automorphism group meeting in a node ${P'}^{(i)}$. For every $i$ there exists a neighbourhood $S^{(i)}=S({\stab_2'}^{(i)},{\spinii'_2}^{(i),\nu},{P'}^{(i),-})$ as above. It is possible that $S^{(i)}=S^{(i')}$ for different indices $i$ and $i'$, but after dropping such dublicates and possibly shrinking the remaining $S^{(i)}$ we may assume that these are disjoint. 

Let $V=\CC[3g-3]_u$, $V^0$ the locus where $G=\paut\spin$ acts freely, $Y=V/G\cup\bigcup_i S^{(i)}$, $\widetilde Y\rightarrow Y$ a desingularisation, $U=V^0/G\cup\left(V/G\cap\bigcup_iS^{(i)}\right)$ and $\widetilde U$ the preimage of $U$ under the desingularisation $\widetilde Y\rightarrow Y$. 

\emph{Claim: $U$ fulfils the assumptions of the generalised \rsbt{} criterion.} $U$ is open in $V/G$ and contains $V^0/G$ by definition. Now let $M\in G$ be an element with \rsbt{} sum $0<\RT<1$ with respect to some primitive $\ord M$th root. By Corollary~\ref{cor:bad} $M$ is the action of an automorphism $\iso$ such that $\stab$ has exactly one elliptic tail on which $\isoi_\stab$ acts of order $3$ or $6$, on all other irreducible components $\isoi_\stab$ is the identity and the theta characteristic on the elliptic tail is trivial. Hence, the elliptic tail of order $3$ or $6$ has to be $\stab_1^{(i)}$ for some $i\in\{1,\dotsc,k\}$. If $u_{2i-1}$ is the coordinate corresponding to the elliptic tail node $P^{(i)}$ and $u_{2i}$ is the coordinate corresponding to deforming the elliptic tail $\stab_1^{(i)}$, then $\fix(M)=\{u_{2i-1}=u_{2i}=0\}$. In the deformed spin curve $({\spini'}^{(i)},{\spinii'}^{(i)},{\spiniii'}^{(i)})$ the elliptic tail $\stab_1^{(i)}$, its elliptic tail node and the theta characteristic on it are preserved, therefore, the moduli point of this spin curve in $\CC[3g-3]_\tau/\aut\spin$ lies in the image of the locus $\{\tau_{2i-1}=\tau_{2i}=0\}\subset\CC[3g-3]_\tau$ and under the identification $\CC[3g-3]_\tau/\aut\spin\cong V/G$ this point lies in the image of $\fix(M)$ in $V/G$. By definition the point also lies in $S^{(i)}$, hence in $U$. In particular, the intersection of $U$ with the image of $\fix(M)$ is non-empty.

As in the second case the restriction of the form $\omega$ to $S^{(i)}\cap\sgbar^{\text{reg}}$ extends holomorphically to the preimage $\widetilde{S^{(i)}}$ under the desingularisation $\widetilde Y\rightarrow Y$. Moreover, $\omega$ is holomorphic on $V^0/G\subset\sgbar^{\text{reg}}$. Therefore, $\omega$ is holomorphic on $\widetilde U$, the assumptions of the generalised \rsbt{} criterion are fulfilled and $\omega$ extends holomorphically to the desingularisation $\widetilde{V/G}\subset\widetilde Y$ of $V/G$.
\end{proof}

{\footnotesize Katharina Ludwig,\\
Institut f\"ur Algebraische Geometrie,\\
Leibniz Universit\"at Hannover\\
Welfengarten 1, D-30167 Hannover, Germany,\\
Email: ludwig (at) math.uni-hannover.de}
\end{document}